\documentclass[runningheads, envcountsame, envcountsect]{llncs}

\usepackage{graphicx}

\usepackage{amsmath}
\usepackage{amssymb}
\usepackage{enumitem}

\usepackage{tikz-cd}
\usetikzlibrary{decorations.markings}

\usepackage{scalerel}

\newcommand{\fr}{\mathfrak{r}}

\newcommand{\fx}{\mathfrak{x}}
\newcommand{\fy}{\mathfrak{y}}

\newcommand{\calC}{\mathcal{C}}

\newcommand{\calL}{\mathcal{L}}

\RequirePackage{mathtools}

\DeclareMathOperator{\ev}{ev}

\DeclareSymbolFont{symbolsA}{U}{txsya}{m}{n}
\DeclareSymbolFont{symbolsC}{U}{txsyc}{m}{n}
\DeclareMathSymbol{\multimapdot}{\mathrel}{symbolsC}{20}
\DeclareMathSymbol{\multimapdotinv}{\mathrel}{symbolsC}{21}
\DeclareMathSymbol{\multimap}{\mathrel}{symbolsA}{40}
\DeclareMathSymbol{\multimapinv}{\mathrel}{symbolsC}{18}

\newcommand{\catA}{\catfont{A}}
\newcommand{\catB}{\catfont{B}}
\newcommand{\catC}{\catfont{C}}

\newcommand{\catY}{\catfont{Y}}

\newcommand{\SET}{\catfont{Set}}
\newcommand{\REL}{\catfont{Rel}}

\newcommand{\ORD}{\catfont{Ord}}
\newcommand{\EQ}{\catfont{Equ}}

\newcommand{\MET}{\catfont{GMet}}
\newcommand{\UMET}{\catfont{GUMet}}
\newcommand{\BMET}{\catfont{BHMet}}
\newcommand{\PBMET}{\catfont{BPMet}}

\newcommand{\Rels}[1]{#1\text{-}\catfont{Rel}}
\newcommand{\Cats}[1]{#1\text{-}\catfont{Cat}}

\newcommand{\Coalg}[1]{\catfont{CoAlg}(#1)}

\newcommand{\LaxF}{\catfont{Lax(\ftF)}}
\newcommand{\LiftF}{\catfont{Lift(\ftF)}}
\newcommand{\LiftFi}{\catfont{Lift(\ftF)_I}}

\newcommand{\PredF}{\catfont{Pred(\ftF)}}
\newcommand{\PredFm}{\catfont{Pred(\ftF)_M}}

\newcommand{\op}{\mathrm{op}}
\newcommand{\ar}{\mathrm{ar}}

\newcommand{\sep}{\mathrm{sep}}

\newcommand{\sym}{\mathrm{sym}}

\newcommand{\two}{\catfont{2}}

\newcommand{\V}{\mathcal{V}}

\newcommand{\catfont}[1]{\mathsf{#1}}

\makeatletter

\def\slashedarrowfill@#1#2#3#4#5{
  $\m@th\thickmuskip0mu\medmuskip\thickmuskip\thinmuskip\thickmuskip
   \relax#5#1\mkern-7mu
   \cleaders\hbox{$#5\mkern-2mu#2\mkern-2mu$}\hfill
   \mathclap{#3}\mathclap{#2}
   \cleaders\hbox{$#5\mkern-2mu#2\mkern-2mu$}\hfill
   \mkern-7mu#4$
}

\newcommand*{\rightrelarrowfill@}{\slashedarrowfill@\relbar\relbar{\raisebox{0pc}{$\mapstochar$}}\rightarrow}
\newcommand*{\xrelto}[2][]{\ext@arrow 0055{\rightrelarrowfill@}{\;#1\;}{\;#2\;}}

\newcommand*{\relto}{\xrelto{}}

\newcommand*{\rightmodarrowfill@}{\slashedarrowfill@\relbar\relbar{\raisebox{0pc}{$\hspace{1pt}\circ$}}\rightarrow}
\newcommand*{\xmodto}[2][]{\ext@arrow 0055{\rightmodarrowfill@}{\;#1\;}{\;#2\;}}

\newcommand*{\modto}{\xmodto{\;}}

\newcommand*{\rightkrelarrowfill@}{\slashedarrowfill@\relbar\relbar{\raisebox{0pc}{$\hspace{1pt}\mapstochar$}}\rightharpoonup}
\newcommand*{\xkrelto}[2][]{\ext@arrow 0055{\rightkrelarrowfill@}{\;#1\;}{\;#2\;}}

\newcommand*{\rightkmodarrowfill@}{\slashedarrowfill@\relbar\relbar{\raisebox{0pc}{$\hspace{1pt}\circ$}}\rightharpoonup}
\newcommand*{\xkmodto}[2][]{\ext@arrow 0055{\rightkmodarrowfill@}{\;#1\;}{\;#2\;}}

\makeatother

\newcommand{\dasht}{\rotatebox[origin=c]{90}{\(\dashv\)}}
\newcommand{\tdash}{\rotatebox[origin=c]{270}{\(\dashv\)}}

\newcommand{\ftD}{\functorfont{D}}
\newcommand{\ftF}{\functorfont{F}}
\newcommand{\ftG}{\functorfont{G}}

\newcommand{\ftI}{\functorfont{I}}

\newcommand{\ftM}{\functorfont{M}}
\newcommand{\ftN}{\functorfont{N}}
\newcommand{\ftP}{\functorfont{P}}

\newcommand{\ftId}{\functorfont{Id}}

\DeclareMathOperator{\bd}{bd}

\newcommand{\ftII}[1]{{\catfont{|}{#1}\catfont{|}}}

\newcommand{\ftbF}{\functorfont{\overline{F}}}

\newcommand{\eF}{\widehat{\ftF}}

\newcommand{\functorfont}{\mathsf}

\RequirePackage{bbm}

\DeclareMathAlphabet{\mathpzc}{OT1}{pzc}{m}{it}

\DeclareMathOperator{\yoneda}{\mathpzc{y}}

\RequirePackage[mathscr]{euscript}

\RequirePackage{dsfont}

\newcommand{\df}[1]{\emph{\textbf{#1}}}

\tikzset{
  relational/.style={
    outer sep=3pt,
    decoration={
      markings,
      mark=at position 0.5 with {\node[transform shape] (tempnode) {\tiny $\rvert$};},
    },
    postaction={decorate},
  },
}

\tikzset{
  distrib/.style={
    outer sep=3pt,
    decoration={
      markings,
      mark=at position 0.5 with {\node[transform shape] (tempnode) {\tiny
          \rmfamily o};},
    },
    postaction={decorate},
  },
}

\tikzset{
  symbol/.style={
    draw=none, every to/.append style={
      edge node={node [sloped, allow upside down,auto=false]{$#1$}}} } }

\newcommand{\one}{\catfont{1}}

\usepackage{comment}

\usepackage{tabularx}

\def\nlabel#1#2{\begingroup #2
\def\@currentlabel{#2}
\phantomsection\label{#1}\endgroup
}

\spnewtheorem{defn}[theorem]{Definition}{\bfseries}{\rmfamily}

\usepackage{stmaryrd}
\newcommand{\Sema}[1]{\llbracket{#1}\rrbracket_\alpha}

\usepackage[pdftex,colorlinks,linkcolor={blue},citecolor={blue},urlcolor={red},breaklinks=true,final]{hyperref}

\usepackage{cleveref}

\usepackage{hypcap}
\usepackage{microtype}

\begin{document}

\title{Kantorovich Functors and Characteristic Logics for Behavioural Distances}

\author{Sergey Goncharov\inst{1}\orcidID{0000-0001-6924-8766}\thanks{Funded by the Deutsche Forschungsgemeinschaft (DFG, German Research Foundation) -- project number 501369690}, Dirk Hofmann \inst{2}\orcidID{0000-0002-1082-6135}\thanks{Funded by The Center for Research and Development in Mathematics and Applications (CIDMA) through the Portuguese Foundation for Science and Technology (FCT) -- project numbers UIDB/04106/2020 and UIDP/04106/2020}, Pedro Nora \inst{1}\orcidID{0000-0001-8581-0675}\thanks{Funded by the Deutsche Forschungsgemeinschaft (DFG, German Research Foundation) -- project number 259234802}, Lutz Schröder \inst{1}\orcidID{0000-0002-3146-5906}\thanks{Funded by the Deutsche Forschungsgemeinschaft (DFG, German Research Foundation) -- project number 434050016} and Paul~Wild \inst{1}\orcidID{0000-0001-9796-9675}}

\institute{Friedrich-Alexander-Universität Erlangen-Nürnberg, Germany \and Center for Research and Development in Mathematics and Applications, University of Aveiro, Portugal}

\authorrunning{S.\ Goncharov et al.}
\maketitle

\begin{abstract}
	Behavioural distances measure the deviation between states in
	quantitative systems, such as probabilistic or weighted systems.
	There is growing interest in generic approaches to behavioural
	distances.  In particular, coalgebraic methods capture variations in
	the system type (nondeterministic, probabilistic, game-based etc.),
	and the notion of \emph{quantale} abstracts over the actual values
	distances take, thus covering, e.g., two-valued equivalences,
	(pseudo)metrics, and probabilistic (pseudo)metrics.  Coalgebraic
	behavioural distances have been based either on \emph{liftings} of
	\(\SET\)-functors to categories of metric spaces, or on \emph{lax
		extensions} of \(\SET\)-functors to categories of quantitative
	relations. Every lax extension induces a functor lifting but not
	every lifting comes from a lax extension.
	It was shown recently that every lax extension is Kantorovich, i.e.\ induced by a
	suitable choice of monotone predicate liftings, implying via a
	quantitative coalgebraic Hennessy-Milner theorem that behavioural
	distances induced by lax extensions can be characterized by
	quantitative modal logics. Here, we essentially show
	the same in the more general setting of behavioural distances
	induced by functor liftings. In particular, we show that every
	functor lifting, and indeed every functor on (quantale-valued)
	metric spaces, that preserves isometries
	is Kantorovich, so that the induced behavioural distance (on systems
	of suitably restricted branching degree) can be characterized by a
	quantitative modal logic.
\end{abstract}

\section{Introduction}

Qualitative transition systems, such as standard labelled transition
systems, are typically compared under two-valued notions of
behavioural equivalence, such as Park-Milner bisimilarity. For
quantitative systems, such as probabilistic, weighted, or metric
transition systems, notions of \emph{behavioural distance} allow for a
more fine-grained comparison, in particular give a numerical measure
of the deviation between inequivalent states, instead of just flagging
them as inequivalent~\cite{GJS90,BW05,AFS09,LarsenEA11}.

The variation found in the mentioned system types calls for unifying
methods, and correspondingly has given rise to generic notions of
behavioural distance based on \emph{universal
	coalgebra}~\cite{Rut00a}, a framework for state-based systems in
which the transition type of systems is encapsulated as an
(endo-)functor on a suitable base category.  Coalgebraic behavioural
distances have been defined on the one hand using \emph{liftings} of
given set functors to the category of metric spaces~\cite{BBKK18}, and
on the other hand using \emph{lax extensions}, i.e.\ extensions of set
functors to categories of quantitative
relations~\cite{Gavazzo18,WS20}. Since every lax extension induces a
functor lifting in a straightforward way~\cite{WS20} but on the other
hand not every functor lifting is induced by a lax extension, the
approach via liftings is more widely applicable. On the other hand, it
has been shown that every lax extension is \emph{Kantorovich}, i.e.\
induced by a suitable choice of modalities, modelled as predicate
liftings in the spirit of coalgebraic logic~\cite{Pat04,Sch08}. Using
quantitative coalgebraic Hennessy-Milner theorems, it follows
that under expected conditions on the functor and the lax extension,
behavioural distance coincides with logical distance.

Roughly speaking, our main contribution in the present paper is to
show that the same holds for functor liftings and their induced
behavioural distances. In more detail, we have the following (cf.~Figure~\ref{fig:big-picture} for a graphical summary):
\begin{itemize}
	\item \emph{Every} lifting of a set functor is \emph{topological},
	      i.e.\ induced by a generalized form of predicate liftings in which
	      one may need to switch to non-standard spaces of truth values for
	      the predicates involved (\autoref{p:54}).
	\item Functor liftings that preserve  isometries are \emph{Kantorovich}, i.e.\ induced by (possibly
	      polyadic) predicate liftings. (Here, we understand predicate
	      liftings as involving only the standard space of truth values --
	      that is, the unit interval, in the case of $1$-bounded metric
	      spaces). In fact, preservation of isometries is also
	      necessary (\autoref{p:53}).
	\item Lastly, we detach the technical development from set functors,
	      and show that a functor on (pseudo)metric spaces is
	      \emph{Kantorovich}, in the sense that the distance of its elements
	      can be characterized by predicate liftings, iff it preserves isometries~(\autoref{thm:kantorovich-initial}).
\end{itemize}
By a recent coalgebraic quantitative Hennessy-Milner theorem that fits
this level of generality~\cite{ForsterEA23}, it follows that given a
functor~$\ftF$ on (pseudo<)metric spaces that preserves isometries, acts non-expansively on morphisms, and admits a dense
finitary subfunctor, behavioural distance can be characterized by
quantitative modal logic (\autoref{cor:unit-hm}). In additional
results, we further clarify the relationship between functor liftings
and lax extensions, and in particular characterize the functor
liftings that are induced by lax extensions (\autoref{p:19}).

Indeed, we conduct the main technical development not only in
coalgebraic generality, but also parametric in a
quantale, hence abstracting both over distances and over truth
values. One benefit of this generality is that our results cover the
two-valued case, captured by the two-element quantale. In particular,
one instance of our results is the fact that every finitary set
functor has a separating set of finitary predicate liftings, and hence
admits a modal logic having the Hennessy-Milner
property~\cite{Sch08}. Moreover, we do
not restrict to symmetric distances, and hence cover also simulation
preorders and simulation distances~\cite{LarsenEA11}.

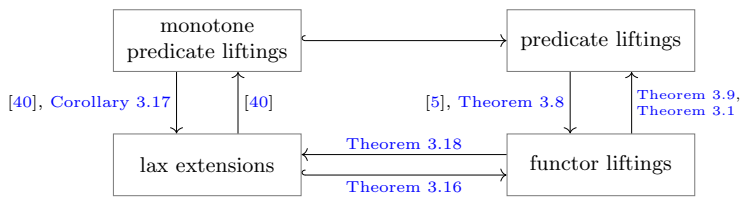
\begin{figure}[htp]
	\[
		\begin{tikzcd}[cells={nodes={draw=gray}},column sep=20ex, row sep=6ex,style={every label/.append style={scale=.9},cells={nodes={scale=.9}}}]
			\parbox[][5ex][c]{2.5cm}{\centering monotone\\ predicate liftings}
			\ar[r,hook]
			\dar["\text{\cite{WS20}, \autoref{p:62}}"',xshift=-3ex] &
			\parbox[][5ex][c]{2.5cm}{\centering predicate liftings}
			\dar[xshift=-3ex,"\text{\cite{BBKK18}, \autoref{p:51}}"']\\
			\parbox[][5ex][c]{2.5cm}{\centering lax extensions\\{}}
			\rar[yshift=-1ex,hook, "\text{\autoref{p:61}}"']
			\uar["\text{\cite{WS20}}"',xshift=3ex] &
			\parbox[][5ex][c]{2.5cm}{\centering functor liftings\\{}}
			\lar[yshift=1ex, "\text{\autoref{p:19}}"']
			\uar[xshift=3ex,"\parbox{5em}{\tiny\autoref{p:53},\\ \autoref{p:54}}"']
		\end{tikzcd}
	\]
	\caption{Summary of connections (a rigorous categorical interpretation of these
		connections involves a square of adjunctions~\eqref{eq:adjunctions}).}
	\label{fig:big-picture}
\end{figure}
\vspace{-2em}
\paragraph*{Related Work} Quantale-valued quantitative notions of
bisimulation for functors that already live on generalized metric
spaces (rather than being lifted from functors on sets) have been
considered early on~\cite{Wor00}. We have already mentioned previous
work on coalgebraic behavioural metrics, for functors originally
living on sets, via metric liftings~\cite{BBKK18} and via lax
extensions~\cite{Gavazzo18,WS20}.
Existing work that combines coalgebraic and quantalic generality and
accommodates asymmetric distances, like the present work, has so far
concentrated on establishing so-called van Benthem theorems, concerned
with characterizing (coalgebraic) quantitative modal logics by
bisimulation invariance~\cite{WS21}. There is a line of work on
Kantorovich-type coinductive predicates at the level of generality of
topological categories~\cite{KKH+19,KKK+21} (phrased in fibrational
	terminology), with results including a game characterization and
expressive logics for coinductive predicates already assumed to be
Kantorovich in a general sense, i.e.~induced by variants of predicate
liftings. In this work, the condition of preserving isometries
already shows up as \emph{fiberedness}, and indeed  the
condition already appears in work on metric liftings~\cite{BBKK18}. As
mentioned in the above discussion, we complement existing work on
quantitative coalgebraic Hennessy-Milner
theorems~\cite{KM18,WS20,ForsterEA23} by establishing the Kantorovich
property they assume.

\section{Preliminaries}

We will need a fair amount of material on coalgebra, quantales and
quantale-enriched categories (generalizing metric spaces), predicate
liftings, and lax extension, which we recall in the sequel. New
material starts in Section~\ref{sec:lif-vs-pl}.

\subsection{Categories and Coalgebras}

We assume basic familiarity with category theory~\cite{AHS90,Awodey10}.
More specifically, we make extensive use of topological categories \cite{AHS90} and quantale-enriched categories \cite{Law73,Kel82,Stu14}.
Recall that a \df{coalgebra} for a functor \(\ftF \colon \catC \to \catC\)
consists of an object~\(X\) of \(\catC\), thought of as an object of
\emph{states}, and a morphism $\alpha \colon X \rightarrow \ftF X$, thought of
as assigning structured collections (sets, distributions, etc.) of successors to
states. A \df{coalgebra morphism} from \((X, \alpha)\) to \((Y, \beta)\) is a
morphism $f \in \catC(X, Y)$ such that \(\beta \cdot f = \ftF f \cdot \alpha\).
We will focus on \df{concrete categories} over \(\SET\), that is categories that
come equipped with a faithful functor $\ftII{-}\colon \catC \rightarrow \SET$,
which allows speaking about individual \emph{states} as elements of $\ftII{X}$.
A \df{lifting} of an endofunctor $\ftF\colon\SET\to\SET$ to $\catC$ is an endofunctor
$\overline{\ftF}\colon\catC\to\catC$ such that $\ftII{-}\cdot\overline\ftF =
	\ftF\cdot\ftII{-}$.
\begin{example}
	\label{p:8}
	Some functors of interest for coalgebraic modelling are as follows.
	\begin{enumerate}[wide]
		\item The \emph{finite powerset functor} \(\ftP_\omega \colon \SET \to \SET\) maps each set to its finite powerset, and for a map~$g$, $\ftP_\omega(g)$ takes direct images under~$g$.
		      Given a set~\(A\) (of \emph{labels)}, coalgebras for the the functor \(\ftP_\omega(A \times -)\)
		      are finitely branching $A$-labelled transition systems.
		\item \label{p:7} The \emph{finite distribution functor} \(\ftD_\omega \colon \SET \to \SET\) maps a set $X$
		      to the set $\ftD_\omega X$ of finitely supported probability distributions on $X$.
		      Given a finite set \(A\), coalgebras for the functor \((1 + \ftD_\omega)^A\), are probabilistic transition 		systems~\cite{DBLP:journals/iandc/LarsenS91,DBLP:journals/iandc/DesharnaisEP02}.
	\end{enumerate}
	Finitary functors are those which are determined by their action on finite sets.
	More precisely, a functor is finitary if for every set \(X\) and
	every \(\fx \in \ftF X\), there is a finite subset inclusion \(m \colon A \to X\) such
	that \(\fx\) is in the image of \(\ftF m\).

	Standard examples of non-finitary functors are as follows.
	\begin{enumerate}[wide,resume]
		\item The (unbounded) powerset functor \(\ftP \colon \SET \to \SET\).
		\item The neighbourhood functor \(\ftN \colon \SET \to \SET\) sends a set \(X\) to the set \(\ftP\ftP
		      X\), and a function \(f \colon X \to Y\) to the function \(\ftN f \colon
		      \ftN X \to \ftN Y\) that assigns to every element \(\fx \in \ftN X\) the set
		      \(\{B \subseteq Y \mid f^{-1} B \in \fx \}\).
	\end{enumerate}
\end{example}

\subsection{Quantales and Quantale-Enriched Categories}
\label{sec:q-rel+cat}

A central notion of our development is that of a quantale, which will
serve as a parameter determining the range of truth values and
distances.  A \df{quantale}~\((\V,\otimes,k)\), more precisely a
commutative and unital quantale, is a complete lattice $\V$ -- with
joins and meets denoted by~$\bigvee$ and $\bigwedge$, respectively --
that carries the structure of a commutative monoid with
\df{tensor}~$\otimes$ and \df{unit}~$k$, such that for every
$u \in \V$, the map $u \otimes - \colon \V\to \V$ preserves suprema.
This entails that every $u \otimes - $ has a right
adjoint $\hom(u,-) \colon \V \to \V$, characterized by the property
\( u \otimes v \le w \iff v \le \hom(u,w).  \) We denote by $\top$ and
$\bot$ the greatest and the least element of a quantale, respectively.
A quantale is \df{non-trivial} if $\bot\neq\top$, and \df{integral} if
\(\top = k\).

\begin{example}
	\label{p:30}
	\begin{enumerate}[wide]
		\item\label{item:frame} Every frame (i.e.\ a complete lattice in which binary meets distribute over
		infinite joins) is a quantale with \(\otimes = \wedge\) and \(k = \top\).
		In particular, every finite distributive lattice is a quantale, prominently~\(\two\), the two-element lattice $\{\bot,\top\}$ and \(\one\), the trivial quantale.
		\item Every left continuous \(t\)-norm \cite{AFS06} defines a quantale on the unit interval equipped with its natural order.
		\item \label{ex:metric-q} The previous clause (up to isomorphism) further specializes as follows:
		      \begin{enumerate}[left=2ex]
			      \item The quantale \([0,\infty]_+ = ([0,\infty], \inf, +, 0)\) of non-negative real numbers with infinity, ordered by the greater or equal relation, and with tensor given by addition.
			      \item The quantale \([0,\infty]_{\max} = ([0,\infty], \inf, \max, 0)\) of non-negative real numbers with infinity, ordered by the greater or equal relation, and with tensor given by maximum.
			      \item \label{p:4} The quantale \([0,1]_\oplus = ([0,1], \inf, \oplus, 0)\) of the unit interval, ordered by the greater or equal order, and with tensor given by truncated addition.
		      \end{enumerate}
		      (Note that the quantalic order here is dual to the standard numeric order).
		\item Every commutative monoid \((M, \cdot, e)\) generates a quantale on \(\ftP M\)
		      (the free quantale over~\(M\)) w.r.t.\ set inclusion and with the tensor
		      \(
		      A \otimes B = \{ a \cdot b \mid a \in A \text{ and } b \in B\},
		      \)
		      for all \(A,B \subseteq M\).
		      The unit of this multiplication is the set \(\{e\}\).
	\end{enumerate}
\end{example}

A \df{\(\V\)-category}
is pair \((X,a)\) consisting of a set $X$ and a map
$a \colon X \times X \to \V$ such that \(k \le a(x,x)\) and
\(a(x,y) \otimes a(y,z) \leq a(x,z)\) for all $x,y,z \in X$. We view
$a$ as a (not necessarily symmetric) \emph{distance} function, noting however that objects with `greater' distance should be seen as being closer together. A \(\V\)-category \((X,a)\) is
\df{symmetric} if \(a(x,y) = a(y,x)\) for all \(x,y \in X\).
Every $\V$-category \((X,a)\) carries a \df{natural order} defined by
\(
x \le y \text{ whenever } k \le a(x,y),
\)
which induces a faithful functor $\Cats{\V} \to \ORD$.
A \(\V\)-category is \df{separated} if its natural order is antisymmetric.
A \df{$\V$-functor} $f \colon (X,a) \to (Y,b)$ is a map $f \colon X \to Y$ such that, for all $x,y \in X$, \( a(x,y) \le b(f(x),f(y)).\)
\(\V\)-categories and \(\V\)-functors form the category \(\Cats{\V}\), and we denote by \(\Cats{\V}_\sym\) the full subcategory of \(\Cats{\V}\) determined by the symmetric \(\V\)-categories and by \(\Cats{\V}_{\sym,\sep}\) the full subcategory of \(\Cats{\V}_{\sym}\) determined by the separated symmetric \(\V\)-categories.

\begin{example}
	\begin{enumerate}[wide]
		\item The Category \(\Cats{\one}\) is equivalent to the category \(\SET\) of \df{sets} and functions.
		\item  The category $\Cats{\two}$ is equivalent to the category $\ORD$ of \df{preordered sets} and monotone maps.
		\item  Metric, ultrametric and bounded metric spaces \`a la Lawvere \cite{Law73} can be seen as quantale-enriched categories:
		      \begin{enumerate}[left=3ex]
			      \item The category \(\Cats{[0,\infty]_+}\) is equivalent to the category \(\MET\) of generalized \df{metric spaces} and non-expansive maps.
			      \item The category \(\Cats{[0,\infty]_{\max}}\) is equivalent to the category $\UMET$ of generalized \df{ultrametric spaces} and non-expansive maps.
			      \item The category $\Cats{[0,1]_{\oplus}}$ is equivalent to the category $\BMET$ of \df{bounded-by-$1$ hemimetric spaces} and non-expansive maps.
		      \end{enumerate}
		\item Categories enriched in a free quantale \(\ftP M\) on a monoid \(M\) can be interpreted as sets equipped with a non-deterministic \(M\)-valued structure.
		\end{enumerate}
\end{example}
\begin{table}[t]
	\begin{center}
		\begin{tabular}{l|>{\raggedright\arraybackslash}p{3.6cm}|>{\raggedright\arraybackslash}p{5cm}}
			General $\V$            & Qualitative ($\V=2$)                   & Quantitative ($\V=[0,1]_\oplus$) \\\hline
			$\V$-category           & preorder                               & bounded-by-1 hemimetric space   \\
			symmetric $\V$-category & equivalence                            & bounded-by-1 pseudometric space \\
			$\V$-functor            & monotone map                           & non-expansive map                \\
			initial $\V$-functor    & order-reflecting\newline monotone  map & isometry                         \\
		\end{tabular}
	\end{center}
	\caption{$\V$-categorical notions in the qualitative and the quantitative setting. The prefix `pseudo' refers to absence of separatedness, and the prefix `hemi' additionally indicates absence
		of symmetry.}
	\label{fig:sep}
\vspace{-2em}
\end{table}
We focus on $\V=\two$ and $\V=[0,1]_\oplus$, which we will use to capture classical (qualitative) and metric (quantitative) aspects of system behaviour, respectively..
Table~\ref{fig:sep} provides some instances of generic quantale-based concepts (either introduced above or to be introduced presently) in these two cases, for further reference.

A $\V$-category~$(X,a)$ is \df{discrete} if $a=1_X$, and \df{indiscrete} if $a(x,y)=\top$ for all $x,y\in X$. The \df{dual} of \((X,a)\) is the
\(\V\)-category \({(X,a)}^\op = (X,a^\circ)\) given by
$a^\circ(x,y)=a(y,x)$. Given a set~$X$ and a \df{structured cone},
i.e.\ a family \((f_i \colon X \to \ftII{(X_i,a_i)})_{i\in I}\) of
maps into $\V$-categories $(X_i,a_i)$, the \df{initial structure}
\(a \colon X \times X \to \V\) on~$X$ is defined by
\( a(x,y) = \bigwedge_{i\in I}a_i(f_i(x),f_i(y)), \) for all
$x,y\in X$.  A cone \(((X,a) \to (X_i,a_i))_{i \in I}\) is said to be
\df{initial} (w.r.t.\ the forgetful functor
\(\ftII{-} \colon \Cats{\V} \to \SET\)) if \(a\) is the initial
structure w.r.t.\ the structured cone
\((X \to \ftII{(X_i,a_i)})_{i \in I}\); a \(\V\)-functor is \df{initial} if
it forms a singleton initial cone.  For every \(\V\)-category \((X,a)\) and
every set \(S\), the \df{\(S\)-power}~\((X,a)^S\) is the
\(\V\)-category consisting of the set of all functions from \(S\) to
\(X\), equipped with the \(\V\)-category structure \([-,-]\) given by
\( [f,g] = \bigwedge_{x \in X} a(f(x),g(x)), \) for all
\(f,g \colon S \to X\).  By equipping its hom-sets with the
substructure of the appropriate power, the category \(\Cats{\V}\)
becames \(\Cats{\V}\)-enriched and, hence, also \(\ORD\)-enriched
w.r.t to the corresponding natural order of \(\V\)-categories.  We say
that an endofunctor on \(\Cats{\V}\) is \df{locally monotone} if it
preserves this preorder.

\begin{remark}\label{rem:sym-sep}
	Let us briefly outline the connections between \(\Cats{\V}\) and \(\Cats{\V}_\sym\), which for real-valued $\V$ correspond to hemimetric and pseudometric spaces, respectively.
	By virtue of the above construction of initial structures, the categories \(\Cats{\V}\) and  \(\Cats{\V}_\sym\) are topological over \(\SET\)~\cite{AHS90};
	in particular, both categories are complete and cocomplete.
	Moreover,  \(\Cats{\V}_\sym\) is a (reflective and) coreflective full subcategory of \(\Cats{\V}\).
	The coreflector \((-)_s \colon \Cats{\V} \to \Cats{\V}_\sym\) is identity on morphisms and sends every \((X,a)\) to its symmetrization, the \(\V\)-category \((X,a_s)\) where \(a_s(x,y)= a(x,y) \wedge a(y,x)\) (keep in mind that in \autoref{p:30}.\ref{ex:metric-q}, the order is the dual of the numeric order).
\end{remark}
Finally, we note that for every quantale $\V$, $(\V,\hom)$ is a $\V$-category, which for simplicity we also denote by~\(\V\).
The following result records two fundamental properties of the \(\V\)-category \(\V\).

\begin{proposition}\label{d:prop:2}
	The \(\V\)-category \(\V=(\V,\hom)\) is injective w.r.t.\ initial morphisms, and for every \(\V\)-category \(X\),
	the cone \((f \colon X \to \V)_{f}\) is initial. 
\end{proposition}

\subsection{Predicate Liftings}
\label{sec:pl}

Given a cardinal \(\kappa\) and a \(\V\)-category \(X\), a \(\kappa\)-ary \df{\(X\)-valued predicate lifting} for a functor \(\ftF \colon \Cats{\V} \to \Cats{\V}\) is a natural transformation
\(
\lambda \colon\Cats{\V}(-,X^\kappa) \to \Cats{\V}(\ftF-,X).
\)
When \(\V\) is the trivial quantale, we identify an \(X\)-valued predicate lifting with a natural transformation
\(
\lambda \colon \SET(-,X^\kappa) \to \SET(\ftF-,X)
\)
via the isomorphism \(\SET \cong \Cats{1}\).
In this case, we are primarily interested in predicate liftings valued in the underlying set of another quantale, and we say that such predicate liftings are \df{monotone} if each of its components is a monotone map w.r.t.\ the pointwise order induced by that quantale.

\begin{remark}
	\label{p:43}
	By the Yoneda lemma, every \(\kappa\)-ary \(X\)-valued predicate lifting for a functor \(\ftF \colon \Cats{\V}\to \Cats{\V}\) is  determined by a \(\V\)-functor \(\ftF X^\kappa \to X\).
	In particular, the collection of all \(X\)-valued \(\kappa\)-ary predicate liftings for a functor is a set.
\end{remark}

\begin{example}
	\label{p:6}
	\begin{enumerate}[wide]
		\item The Kripke semantics of the standard diamond modality \(\Diamond\) of the modal logic \(K\) is induced (in a way recalled in Section~\ref{sec:log}) by the unary predicate lifting $\Diamond_X(A) = \{B \subseteq X \mid A \cap B \neq \varnothing\}$ for the (finite) powerset functor (modulo the isomorphism $\mathcal{P} X\cong\SET(X,\two)$).
		\item \label{p:10} Computing the expected value for a given $[0,1]$-valued function with respect to each probability distribution defines a unary \([0,1]\)-valued predicate lifting for the functor \(\ftD_\omega \colon \SET \to \SET\), which we denote by \(\mathbb{E}\).
	\end{enumerate}
\end{example}

\subsection{Quantale-Enriched Relations and Lax Extensions}
\label{sec:qr+le}

The structure of a quantale-enriched category is a particular kind of ``enriched relation''.
For a quantale $\V$ and sets $X$ and $Y$, a \df{$\V$-relation} from $X$ to $Y$ is a map $r\colon X \times Y \to \V$; we then write $r\colon X \relto Y$.
As for ordinary relations, a pair of $\V$-relations \(r \colon X \relto Y\) and \(s \colon Y \relto Z\) can be composed via ``matrix multiplication'':
\(
(s \cdot r)(x,z)=\bigvee_{y \in Y} r(x,y) \otimes s(y,z)
\)
for \(x \in X\), \(z \in Z\).
With this composition, the collection of all sets and \(\V\)-relations between them form a category, denoted~\(\Rels{\V}\).
The identity morphism on a set \(X\) is the \(\V\)-relation \(1_X \colon X \relto X\) that sends every diagonal element to \(k\) and all the others to \(\bot\).

\begin{example}
	The category of $2$-relations is the usual category \(\REL\) of sets and relations.
	Quantitative or ``fuzzy'' relations are usually defined as \([0,1]_\oplus\)-relations (e.g.~\cite{WS20,BBKK18}).
\end{example}
The category \(\Rels{\V}\) comes with an involution
\(
{(-)}^\circ \colon {\Rels{\V}}^\op \to \Rels{\V}
\)
that maps objects identically and sends a \(\V\)-relation \({r \colon X \relto Y}\) to the \(\V\)-relation \(r^\circ \colon Y \relto X\) given by \(r^\circ(y,x) = r(x,y)\), the \df{converse} of~\(r\).
Moreover, by equipping its hom-sets with the pointwise order induced by \(\V\), \(\Rels{\V}\) is made into a quantaloid (e.g.~\cite{Rosenthal95}), i.e.\ enriched over complete join semilattices.
This entails that there is an optimal way of extending a \(\V\)-relation \(r \colon X \relto Y\) along a \(\V\)-relation \(s \colon X \relto Z\):
the (Kan) \df{extension} of \(r\) along \(s\) is the \(\V\)-relation  \(r \multimapdotinv s \colon Z \relto Y\)  defined by the property
\(
t \cdot s \leq r \iff t \leq r \multimapdotinv s,
\)
for all \(t \colon Z \relto Y\).

A \df{lax extension}
\footnote{
	Extensions of \(\SET\)-functors to \(\REL\) are also commonly referred to as ``relators'', ``relational liftings'' or ``lax relational liftings''.
}
of a functor \(\ftF \colon \SET \to \SET\) to \(\Rels{\V}\) is a lax functor \(\eF \colon \Rels{\V} \to \Rels{\V}\) that agrees with \(\ftF\) on sets and whose action on functions is compatible with \(\ftF\).
To make the latter requirement precise, we note that a function is interpreted as the \(\V\)-relation that sends every element of its graph to \(k\) and all the others to \(\bot\);
then, a lax extension of \(\ftF\) to \(\Rels{\V}\), or simply a lax extension, is a map
\(
(r \colon X \relto Y) \longmapsto (\eF r \colon \ftF X \relto \ftF Y)
\)
such that:
\begin{description}
	\item[\nlabel{p:102}{(L1)}] $r \le r' \implies \eF r \le \eF r'$,
	\item[\nlabel{p:26}{(L2)}] $\eF s\cdot\eF r\le\eF (s\cdot r)$,
	\item[\nlabel{p:0}{(L3)}] $\ftF f \le \eF f$ and
		${(\ftF f)}^\circ\le\eF(f^\circ)$,
\end{description}
for all \(r \colon X \relto Y\), \(s \colon Y \relto Z \) and \(f \colon X \to Y\).

\begin{example}
	\label{p:3}
	The generalized ``lower-half'' Egli-Milner order between powersets, which for a relation \(r \colon X \relto Y\) is defined as the relation \(\widehat{\ftP}r \colon \ftP X \relto \ftP Y\) given by
	\[
		A (\widehat{\ftP}r) B \iff \forall a \in A.\, \exists b \in B.\, a \mathrel{r} b,
	\]
	defines a lax extension of the powerset functor \(\ftP \colon \SET \to \SET\) to \(\REL\).
	Similarly, the generalized ``upper-half'' and the generalized Egli-Milner order define lax extensions of the powerset functor to \(\REL\).
\end{example}
Lax extensions are deeply connected with monotone predicate liftings.
To realize this, it is convenient to think of the \(X\)-component of a \(\kappa\)-ary predicate lifting as a map of type
\(
\Rels{\V}(\kappa,X) \to \Rels{\V}(1,\ftF X)
\)
\cite{GoncharovEA22}.
\footnote{Note that Goncharov et. al. consider as their main point of view the dual of the one considered here \cite[Proposition~4.2]{GoncharovEA22}.
	Our choice prevents a harmless mismatch between the Kantorovich liftings and Kantorovich extensions in Theorem~\ref{p:53}.}
\begin{defn}
  \label{p:11}
  A \(\kappa\)-ary predicate lifting \(\lambda\) for a functor
  \(\ftF \colon \SET \to \SET\) is \df{induced} by a lax extension
  \(\eF \colon \Rels{\V}\! \to\! \Rels{\V}\) if there is a \(\V\)-relation
  \(\fr \colon 1\! \relto\! \ftF \kappa\)
  such that
   \(\lambda(f) = \eF f \cdot \fr \),
  for every \(\V\)-relation
  \(f \colon \kappa \relto X\).
\end{defn}

\begin{example}
	By interpreting a subset of a set \(X\) as a relation from \(1\) to \(X\), the unary predicate lifting \(\Diamond\) (see Example~\ref{p:6}) for the powerset functor \(\ftP \colon \SET \to \SET\) is induced by the lax extension of Example~\ref{p:3}; indeed, it is determined by the map \(1 \to \ftP 1\) that selects the set \(1\).
\end{example}
\begin{remark}
	\label{p:5}
	Every predicate lifting induced by a lax extension is monotone.
\end{remark}
Lax extensions have been instrumental in coalgebraic notions of \emph{behavioural distance} (e.g. \cite{Gavazzo18,WS20,WS21}), and the notion of Kantorovich extension has been crucial to connect such notions with coalgebraic modal logic \cite{CKP+11}.
\begin{defn}
	Let \(\ftF \colon \SET \to \SET\) be a functor, and \(\Lambda\) a \emph{class} of monotone predicate liftings for \(\ftF\).
	The \df{Kantorovich} lax extension of \(\ftF\) w.r.t. \(\Lambda\) is the lax extension
	\(
	\eF^\Lambda = \bigwedge_{\lambda \in \Lambda} \eF^\lambda,
	\)
	where for every \(\V\)-relation \(r \colon X \relto Y\), the \(\V\)-relation \(\eF^\lambda r \colon \ftF X \relto \ftF Y\) given by
	\(
	\label{p:44}
	\eF^\lambda r
	= \bigwedge_{g \colon \kappa \relto X} \lambda(r \cdot g) \multimapdotinv \lambda (g).
	\)
\end{defn}
\begin{example}
	\label{p:47}
	The Kantorovich extension of the powerset functor \(\ftP \colon \SET \to \SET\) to \(\REL\) w.r.t the \(\Diamond\) predicate lifting coincides with the extension given by the ``lower-half'' of the Egli-Milner order (Example~\ref{p:3}).
\end{example}
As suggested by the previous example, the Kantorovich extension leads to a representation theorem that plays an important role in Section~\ref{sec:lax-vs-lif}.
\begin{theorem}[\cite{GoncharovEA22}]
	\label{d:thm:1}
	Let \(\eF \colon \Rels{\V} \to \Rels{\V}\) be a lax extension, and let \(\Lambda\) be the class of all predicate liftings induced by \(\eF\).
	Then, \(\widehat \ftF = \eF^\Lambda\).
\end{theorem}
\section{Topological Liftings}
\label{sec:lif-vs-pl}

It is well-known that every lax extension \(\eF \colon \Rels{\V} \to \Rels{\V}\) of a functor \(\ftF \colon \SET \to \SET\) gives rise to a lifting (which we denote by the same symbol) of $\ftF$ to $\Cats{\V}$ (for instance, see \cite{Tho09}).
By definition, liftings are completely determined by their action on objects.
In particular, the \df{lifting induced by a lax extension} \(\eF \colon \Cats{\V}\to\Cats{\V}\) sends a \(\V\)-category \((X,a)\) to the \(\V\)-category \((\ftF X,\eF a)\).
Of course, it does not make sense to talk about functor liftings to the category \(\Cats{\V}\) when \(\V\) is trivial, hence we assume from now on that \emph{$\V$ is non-trivial}.

Predicate liftings also induce functor
liftings, via a simple construction available on all topological
categories that goes back, at least, to work in categorical duality theory
\cite{DT89,PT91}: To lift a functor
\(\ftG \colon\catA\to\catY\) along a topological functor
\(\ftII{-} \colon\catB\to\catY\), it is enough to give, for every
object \(A\) in \(\catA\), a structured cone
\begin{equation}
	\label{d:eq:4}
	\mathcal{C}(A)=(\ftG A\xrightarrow{h}\ftII{B})_{h,B}
\end{equation}
so that, for every \(h\) in \(\mathcal{C}(A)\) and every
\(f \colon A'\to A\), the composite \(h\cdot\ftG f\) belongs to the
cone \(\mathcal{C}(A')\).  Then, for an object~\(A\) in \(\catA\), one
defines \(\ftG^I A\) by equipping \(\ftG A\) with the initial
structure w.r.t. the structured cone \eqref{d:eq:4}.  It is easy to
see that the assignment \( X \mapsto \ftG^I X \) indeed defines a
functor \(\ftG^I \colon\catA\to\catB\) such that
\(\ftII{-}\cdot\ftG^{I}=\ftG\).  This technique has been previously
applied in the context of \emph{codensity liftings}
\cite{KKH+19,KKK+21,KDH21,Katsumata05} and \emph{Kantorovich liftings}
\cite{BBKK18}.  We apply this to our situation as follows. Given a
functor \(\ftF \colon \SET \to \SET\), take \(\ftG=\ftF\cdot \ftII{-}\); then a lifting of $\ftF$ to \(\Cats{\V}\) can be specified by a \emph{class} of natural transformations
\begin{equation}\label{eq:gen-pl}
	\lambda \colon \Cats{\V}(-,A_\lambda) \longrightarrow \SET(\ftF\ftII{-},\ftII{B_\lambda}),
\end{equation}
(which may be thought of as generalized predicate liftings, in that they lift $A_\lambda$-valued predicates to $B_\lambda$-valued ones).
Namely, given a \(\V\)-category \(X\), we consider the structured cone consisting of all maps
\[
	\lambda(f) \colon \ftF \ftII{X} \longrightarrow \ftII{B_\lambda}
\]
where \(\lambda\) ranges over the given natural transformations
and~$f$ over all \(\V\)-functors \(X \to A_\lambda\).  As described
above, we obtain a \(\V\)-category \(\ftF^I X\) by equipping
\(\ftF |X|\) with the initial structure w.r.t. this cone.  We call
functor liftings constructed in this way \df{topological}. Indeed, it
turns out that \emph{every} functor lifting is topological, even when
one restricts $B_\lambda$ in \eqref{eq:gen-pl} to be the $\V$-category
\((\V,\hom)\):
\begin{theorem}
	\label{p:54}
	Every lifting of a \(\SET\)-functor to \(\Cats{\V}\) is topological w.r.t. a \emph{class} of natural transformations
	\(
	\lambda \colon \Cats{\V}(-,A_\lambda) \longrightarrow \SET(\ftF\ftII{-},\ftII{\V}).
	\)
\end{theorem}
\noindent In examples, we usually construct a generalized predicate
lifting~\eqref{eq:gen-pl} from a~$\kappa$-ary predicate
lifting~$\lambda$ for the set functor~$\ftF$: Choose a pair \((A,B)\)
of \(\V\)-categories over the sets \(\V^\kappa\) and \(\V\),
respectively (the above theorem allows restricting to $B=\V$, and the
examples we present are of this kind). We can then precompose
$\lambda$ with the inclusion natural transformation
$\Cats{\V}(-,A) \longrightarrow \SET(\ftII{-},\ftII{A})$, obtaining a
natural transformation
\(\lambda^{(A,B)} \colon \Cats{\V}(-,A) \to
\SET(\ftF\ftII{-},\ftII{B})\) that applies \(\lambda\) to maps
underlying \(\V\)-functors with codomain \(A\).
\begin{example}
	\label{p:21}
	\begin{enumerate}[wide]
		\item The discrete lifting of the identity functor \(\ftId \colon \SET \to \SET\), which sends every \(\V\)-category to the discrete \(\V\)-category with the same underlying set, can be obtained as a topological lifting constructed from the identity \(\V\)-valued predicate lifting for \(\ftId\) by choosing \(A\) to be the \(\V\)-category consisting of the set \(\V\) equipped with the indiscrete structure.
		\item \label{p:23} The lifting of the identity functor \(\ftId \colon \SET \to \SET\) to \(\ORD\) that computes the smallest equivalence relation that contains a given preorder can be obtained as a topological lifting constructed from the \(\two\)-valued identity predicate lifting for \(\ftId\) by choosing \(A\) to be the discrete preordered set with two elements.
		\item It is well-known that the total variation distance between finite distributions \(\mu,\upsilon\) on a set \(X\) coincides with the Kantorovich distance on the discrete bounded-by-\(1\) metric space \(X\) (e.g. \cite{GS02});
		      that is, \(d^{TV}(\mu,\upsilon) = \bigvee_{f \colon X \to [0,1]} \mathbb{E}_X(f)(\upsilon) \ominus \mathbb{E}_X(f)(\mu)\) (see Example~\ref{p:6}(\ref{p:10})).
		      Therefore, the total variation distance defines a lifting of the finite distribution functor to \(\BMET\) that can be obtained as the topological lifting constructed from the predicate lifting \(\mathbb{E}\) by choosing \(A\) to be the indiscrete space \([0,1]\).
		      This example is closely related to the first one. Indeed, this lifting is the composite of the Kantorovich lifting of the finite distribution functor to \(\BMET\) (see Example~\ref{p:9}) and the discrete lifting of the identity functor to \(\BMET\). By Theorem~\ref{p:53} below,  precomposing functor liftings with the discrete lifting of the identity functor can be used to derive non-Kantorovich s.
	\end{enumerate}
\end{example}

\begin{remark}
	\label{p:12}
	Theorem~\ref{p:54} can be fine-tuned to show that the discrete lifting \(\ftF^d \colon \ORD \to \ORD\) of a finitary functor \(\ftF \colon \SET \to \SET\) is a topological lifting constructed from a set \(\Lambda\) of finitary \(2\)-valued predicate liftings for \(\ftF\).
	Hence, for every set \(X\), considered as a discrete preordered set, we have that the cone of all maps
	\(
	\lambda(f) \colon \ftF^d(X,1_X) \to 2,
	\)
	for \(\kappa\)-ary predicate liftings \(\lambda\in\Lambda\) and maps \(X \to 2^\kappa\), is initial.
	Thus, as \(\ftF^d(X,1_X)\) is antisymmetric, this cone is mono.
	In this sense, our results subsume the result that every finitary \(\SET\)-functor admits a separating set of finitary predicate liftings~\cite{Sch08}.
\end{remark}

\subsection{Kantorovich Liftings}

For our present purposes, we are primarily interested in topological
liftings induced by predicate liftings in the standard sense,
i.e.\ the natural transformations~\eqref{eq:gen-pl} are of
the shape
\( \lambda \colon \Cats{\V}(-,\V^\kappa) \longrightarrow
\SET(\ftF\ftII{-},\ftII{\V}) \), and thus employ~$\V$, equipped with
its standard $\V$-category structure, as the object of truth values
throughout. In particular, this format is needed to use predicate
liftings as modalities in existing frameworks for quantitative
coalgebraic logic (Section~\ref{sec:log}). Many functor liftings
considered in work on coalgebraic behavioural distance can be
understood as topological liftings constructed in this way
(e.g. \cite{BBKK18,KKK+21,WS20,WS21,ForsterEA23}).
To simplify notation, in the sequel we often omit the forgetful functor to \(\SET\).
\begin{defn}
	\label{p:49}
	Let \(\ftF \colon \SET \to \SET\) be a functor and \(\Lambda\) a class of \(\V\)-valued predicate liftings for \(\ftF\).
	The \df{Kantorovich lifting} of \(\ftF\) w.r.t. \(\Lambda\) is the topological lifting \(\ftF^\Lambda \colon \Cats{\V} \to \Cats{\V}\) that sends a \(\V\)-category \(X\) to the \(\V\)-category \((\ftF X, \ftF^\Lambda a)\), where \(\ftF^\Lambda a\) denotes the initial structure on \(\ftF X\) w.r.t. the structured cone of all functions
	\[
		\lambda(f) \colon \ftF \ftII{X} \longrightarrow \ftII{\V}
	\]
	where \(\lambda\in\Lambda\) is \(\kappa\)-ary and \(f \colon (X,a) \to \V^\kappa\) is a \(\V\)-functor.
	Generally, a lifting \(\ftbF \colon \Cats{\V} \to \Cats{\V}\) of $\ftF$ is  \df{Kantorovich} if $\ftbF=\ftF^\Lambda$ some class~$\Lambda$ of predicate liftings for \(\ftF\).
\end{defn}
\begin{example}
	\label{p:9}
	As the name suggests, the prototypical example of a Kantorovich lifting is given by the (non-symmetric) Kantorovich distance between finite distributions, which arises as the Kantorovich lifting of the finite distribution functor on \(\SET\) to the category \(\BMET\) w.r.t the predicate lifting \(\mathbb{E}\) that computes expected values, i.e. \(\ftD_\omega^\mathbb{E}(X,a)(\mu,\upsilon) = \bigvee_{f \colon (X,a) \to [0,1]} \mathbb{E}_X(f)(\upsilon) \ominus \mathbb{E}_X(f)(\mu)\).
\end{example}

We go on to exploit the universal property of initial lifts of cones to characterize the liftings that are Kantorovich.
In the following, fix a functor \(\ftF \colon \SET \to \SET\) and a quantale \(\V\). Consider the partially ordered conglomerate \(\PredF\) of \emph{classes} of \(\V\)-valued predicate liftings for \(\ftF\) ordered by containment, i.e.
\(
\Lambda \leq\Lambda' \iff\Lambda \supseteq \Lambda';
\)
and the partially ordered class \(\LiftF\) of liftings of \(\ftF\) to \(\Cats{\V}\) ordered pointwise, i.e.
\(
\ftbF \leq \ftbF'
\iff \ftbF a \leq \ftbF' a,
\)
for every \(\V\)-category \((X,a)\).
\begin{defn}
	Let \(\ftbF \colon \Cats{\V} \to \Cats{\V}\) be a lifting of \(\ftF\).
	A \(\kappa\)-ary \(\V\)-valued \df{predicate lifting} \(\lambda\) for \(\ftF\) is \df{compatible with} \(\ftbF\) if it restricts to a predicate lifting for \(\ftbF\):
	\begin{displaymath}
		\begin{tikzcd}[row sep=2ex]
			\Cats{\V}(-,\V^\kappa) & \Cats{\V}(\ftbF -,\V) \\
			\SET(-,\ftII{\V^\kappa})& \SET(\ftF\ftII{-}, \ftII{\V})
			\ar[from=1-1, to=1-2, "\lambda", dotted]
			\ar[from=1-1, to=2-1, ]
			\ar[from=1-1, to=2-2, "=", phantom]
			\ar[from=1-2, to=2-2, ]
			\ar[from=2-1, to=2-2, "\lambda"']
		\end{tikzcd}
	\end{displaymath}
	where the vertical arrows denote set inclusions -- that is, if $\lambda$ lifts $\V$-functorial predicates on~$X$ to $\V$-functorial predicates on $\ftbF X$.
	The class of all predicate liftings compatible with \(\ftbF\) is denoted by \(\ftP(\ftbF)\).
\end{defn}
\begin{proposition}
	\label{p:2}
	A \(\kappa\)-ary \(\V\)-valued predicate lifting \(\lambda\) for \(\ftF\) is compatible with~\(\ftbF\) iff the map \(\lambda(1_{\ftII{\V^\kappa}}) \colon \ftF (\ftII{\V^\kappa}) \to \ftII{\V}\) is a \(\V\)-functor of type \(\ftbF \V^\kappa \to \V\).
\end{proposition}
The Kantorovich lifting defines a universal construction:
\begin{theorem}
	\label{p:51}
	Let \(\ftF \colon \SET \to \SET\) be a functor.
	Assigning to a class of predicate liftings for \(\ftF\) the corresponding Kantorovich lifting yields a right adjoint \(\ftF^{(-)} \colon \PredF \to \LiftF\) whose left
	adjoint \(\ftP \colon \LiftF \to \PredF\) maps a lifting of\/ \(\ftF\) to the class \(\ftP(\ftbF)\) of all \(\V\)-valued predicate liftings for \(\ftF\) that are compatible with the lifting.
\end{theorem}
The following result shows that Kantorovich liftings are characterized by a pleasant property that is required in multiple results in the context of coalgebraic approaches to \emph{behavioural distance} (e.g. \cite{BBKK18,KKK+21,ForsterEA23,Wor00}).
\begin{theorem}
	\label{p:53}
	A lifting of a \(\SET\)-functor to \(\Cats{\V}\) is Kantorovich iff it preserves initial morphisms.
\end{theorem}
\begin{corollary}
	\label{p:201}
	Every topological lifting of a functor \(\ftF \colon \SET \to \SET\) w.r.t. a class of natural transformations \(\lambda \colon \Cats{\V}(-,A_\lambda) \to \SET(\ftF -, \ftII{B_\lambda})\) where each \(A_\lambda\) is injective in \(\Cats{\V}\) w.r.t. initial morphisms is Kantorovich.
\end{corollary}
\begin{corollary}
	\label{p:210}
	The composite of Kantorovich liftings is Kantorovich.
\end{corollary}
\begin{example}
	\label{p:103}
	The characterization of Theorem~\ref{p:53} makes it easy to distinguish Kantorovich liftings.
	\begin{enumerate}[wide]
		\item \label{p:120} It is an elementary fact that every lifting induced by a lax extension preserves initial morphisms (e.g.~\cite[Proposition~2.16]{HN20}).
		      In particular, the Wasserstein lifting \cite{BBKK18} is Kantorovich.
		\item \label{p:104} 
		      The identity functor on \(\SET\)  has a lifting
		      \(
		      (-)^\circ \colon \Cats{\V} \to \Cats{\V}
		      \)
		      that sends every \(\V\)-category to its dual.
		      Clearly, this lifting preserves initial morphisms, and hence it is Kantorovich. Indeed, one can show that it is the Kantorovich lifting of the identity functor w.r.t. the set of \(\V\)-valued predicate liftings determined by the representable \(\V\)-functors \(\V^\op \to \V\).
		\item \label{p:13} The functor \((-)_s \colon \Cats{\V} \to \Cats{\V}_\sym\) that symmetrizes \(\V\)-categories gives rise to a lifting \((-)_s \colon \Cats{\V} \to \Cats{\V}\) of the identity functor on \(\SET\).
		      Clearly, this functor preserves initial morphisms, and hence it is Kantorovich. Indeed, one can show that it is the Kantorovich lifting of the identity functor w.r.t. the set of all \(\V\)-valued predicate liftings determined by the representable \(\V\)-functors \(\V_s \to \V\).
		\item The discrete lifting of the identity functor on \(\SET\) to \(\Cats{\V}\) is \emph{not} Kantorovich, as it fails to preserve initial morphisms.
		\item The lifting of the identity functor on \(\SET\) to \(\Cats{\V}\) that sends a \(\V\)-category \((X,a)\) to the \(\V\)-category given by the final structure w.r.t. the structured cospan of identity maps \(\ftII{(X,a)} \rightarrow X \leftarrow \ftII{(X,a^\circ)}\) is \emph{not} Kantorovich.
		      This lifting generalizes Example~\ref{p:21}(\ref{p:23}).
		\item The lifting of the finite distribution functor on \(\SET\) to \(\BMET\) given by the Kantorovich distance is Kantorovich, while the lifting given by the total variation distance is \emph{not} Kantorovich.
	\end{enumerate}
\end{example}
\subsection{Liftings Induced by Lax Extensions}
\label{sec:lax-vs-lif}
We  show next that lax extensions, functor liftings, and predicate liftings are linked by adjunctions, and characterize the liftings induced by lax extensions.
We begin by showing that the Kantorovich extension and the Kantorovich lifting are compatible.
\begin{theorem}
	\label{p:57}
	Let \(\eF \colon \Cats{\V} \to \Cats{\V}\) be a lifting of a functor \(\ftF \colon \SET \to \SET\)
	induced by a lax extension \(\eF \colon \Rels{\V} \to \Rels{\V}\).
	If \(\eF \colon \Rels{\V} \to \Rels{\V}\) is the Kantorovich extension w.r.t. a \emph{class} \(\Lambda\) of predicate liftings, then the functor \(\eF \colon \Cats{\V} \to \Cats{\V}\) is the Kantorovich lifting of \(\ftF \colon \SET \to \SET\) w.r.t. \(\Lambda\).
\end{theorem}

Let \(\LaxF\) denote the partially ordered class of lax extensions of a functor \(\ftF \colon \SET \to \SET\) to \(\Rels{\V}\) ordered pointwise:
\[
	\eF \leq \eF'
	\iff \forall r \in \Rels{\V}. \; \eF r \leq \eF' r;
\]
let \(\LiftFi\) denote the partially ordered \emph{subclass} of \(\LiftF\) consisting of the liftings that preserve initial morphisms, and let \(\PredFm\) denote the partially ordered \emph{subconglomerate} of \(\PredF\) of monotone predicate liftings.
Clearly, the operations of taking Kantorovich extensions
\(
\eF^{(-)} \colon \PredFm \to \LaxF,
\)
and inducing liftings from lax extensions
\(
\ftI \colon \LaxF \to \LiftFi
\)
define monotone maps.
Moreover, as we have seen in Theorem~\ref{p:53}, the monotone map \(\ftF^{(-)} \colon \PredF \to \LiftF\) corestricts to \(\LiftFi\).
Therefore, our results so far tell us that lax extensions, liftings and predicate liftings are connected through a diagram of monotone maps
\begin{small}
	\begin{displaymath}
		\begin{tikzcd}[row sep=large, column sep=scriptsize]
			\LaxF  & \LiftFi \\
			\PredFm & \PredF
			\ar[from=1-1, to=1-2, "\ftI"]
			\ar[from=1-2, to=2-2, "\ftP", bend left]
			\ar[from=2-1, to=2-2, hookrightarrow]
			\ar[from=2-1, to=1-1, "\eF^{(-)}"]
			\ar[from=2-2, to=1-2, "\ftF^{(-)}", bend left]
			\ar[from=2-2, to=1-2, "\vdash", phantom]
		\end{tikzcd}
	\end{displaymath}
\end{small}%
which commutes if the left adjoint is ignored.
In the sequel, we will see that every monotone map in this diagram is an adjoint.
In particular, it might not be immediately obvious that the monotone map \(\eF^{(-)} \colon \PredFm \to \LaxF\) is a right adjoint without  first thinking in terms of functor liftings induced by lax extensions, because the obvious guess -- taking the predicate liftings induced by a lax extension (\autoref{p:11}) -- in general does not define a monotone map \(\LaxF \to \PredFm\).
The next example illustrates this as well as the fact that there are predicate liftings compatible with a functor lifting induced by a lax extension that are not induced by the lax extension.
\begin{example}
	\label{p:110}
	The identity functor on \(\ORD\) is the lifting induced by the identity functor on \(\REL\) as a lax extension of the identity functor on \(\SET\).
	The constant map into \(\top\) is a monotone map \(2 \to 2\) and, hence, determines a predicate lifting that is compatible with the identity functor on \(\ORD\).
	It is easy to see that this predicate lifting is induced by the largest extension of the identity functor, however, it is not induced by the identity functor on \(\REL\)~\cite[Example~3.12]{GoncharovEA22}.
\end{example}
It should also be noted that the predicate liftings compatible with a functor lifting that preserves initial morphisms are not necessarily monotone.
That is, the map \(\ftP \colon \LiftFi \to \PredF\) does not necessarily corestrict to \(\PredFm\).
\begin{example}
	\label{p:76}
	Consider the lifting \({(-)}^\circ \colon \ORD \to \ORD\) of the identity functor on \(\SET\) that sends each preordered set to its dual.
	Then, the predicate lifting for \({(-)}^\circ\) determined by the \(\V\)-functor \(\hom(-,0) \colon {(2,\hom)}^\op \to (2,\hom)\) is not monotone since it sends the constant map \(0 \colon 1 \to 2\) to the constant map \(1 \colon 1 \to 2\).
\end{example}
Accordingly, we need to ``filter the monotone predicate liftings'' first.
This operation trivially defines the left adjoint
\(
\ftM \colon \PredF \to \PredFm
\)
of the inclusion map \(\PredFm \hookrightarrow \PredF\).
\begin{theorem}
	\label{p:61}
	Let \(\ftF \colon \SET \to \SET\) be a functor.
	The monotone map \(\ftI \colon \LaxF \to \LiftFi\) is order-reflecting and right adjoint to the monotone map \(\eF^{\ftM\ftP(-)} \colon \LiftFi \to \LaxF\).
\end{theorem}
\begin{corollary}
	\label{p:62}
	Let \(\ftF \colon \SET \to \SET\) be a functor.
	The monotone map \(\eF^{(-)} \colon \PredFm \to \LaxF\) is right adjoint to the order-reflecting monotone map \(\ftM\ftP\ftI \colon \LaxF \to \PredFm\).
\end{corollary}
Therefore, the interplay between lax extensions, liftings and predicate liftings is captured by the diagram
\begin{equation}\label{eq:adjunctions}
	\begin{tikzcd}[row sep=large, column sep=normal]
		\LaxF  & \LiftFi \\
		\PredFm & \PredF
		\ar[from=1-1, to=1-2, "\ftI"']
		\ar[from=1-1, to=2-1, "\ftM\ftP\ftI"', bend right]
		\ar[from=1-1, to=2-1, "\dashv"', phantom]
		\ar[from=1-2, to=2-2, "\ftP", bend left]
		\ar[from=1-2, to=1-1, "\eF^{\ftM\ftP(-)}"', bend right]
		\ar[from=1-2, to=1-1, "\tdash"', phantom, bend right=12]
		\ar[from=2-1, to=1-1, "\eF^{(-)}"', bend right]
		\ar[from=2-1, to=2-2, hookrightarrow]
		\ar[from=2-2, to=1-2, "\ftF^{(-)}", bend left]
		\ar[from=2-2, to=1-2, "\vdash", phantom]
		\ar[from=2-2, to=2-1, "\ftM", bend left]
		\ar[from=2-2, to=2-1, "\dasht", phantom, bend left=12]
	\end{tikzcd}
\end{equation}
which commutes when only the right adjoints or only the left adjoints are considered.
Finally, we characterize the liftings induced by lax extensions.
\begin{theorem}
	\label{p:19}
	A lifting~$\eF$ of a \(\SET\)-functor~$\ftF$ to \(\Cats{\V}\) is induced by a lax extension of $\ftF$ to \(\Rels{\V}\)  iff~$\eF$ preserves initial morphisms and is locally monotone.
\end{theorem}
\(\V\)-enriched lax extensions have proved to be crucial to deduce quantitative van Benthem and Hennessy-Milner theorems \cite{WS20,WS21}.
We recall that a lax extension of a functor \(\ftF \colon \SET \to \SET\) to \(\Rels{\V}\) is \df{\(\V\)-enriched} \cite{WS21,GoncharovEA22} if, for all \(u \in \V\), \(u \otimes 1_{\ftF X} \leq \eF(u \otimes 1_X)\); where \(u \otimes r\) denotes the \(\V\)-relation ``\(r\) scaled by \(u\)'', that is, \((u \otimes r)(x,y) = u \otimes r(x,y)\).
\begin{theorem}
	\label{p:1}
	A lifting~$\eF$ of a \(\SET\)-functor~$\ftF$ to \(\Cats{\V}\) is induced by a \(\V\)-enriched lax extension of $\ftF$ to \(\Rels{\V}\) iff~$\eF$ preserves initial morphisms and is \(\Cats{\V}\)-enriched.
\end{theorem}
Our characterization of lax extensions makes it clear that there is a large collection of Kantorovich liftings that are not induced by lax extensions.
For instance, it follows from Theorem~\ref{p:19} that the liftings \((-)^\circ \colon \Cats{\V} \to \Cats{\V}\) and \((-)_s \colon \Cats{\V} \to \Cats{\V}\) (see Example~\ref{p:103}) of the identity functor on \(\SET\) to \(\Cats{\V}\) are Kantorovich but are not induced by lax extensions.
Furthermore, as the composite of Kantorovich liftings is Kantorovich, in many situations it is possible to compose these functors with other Kantorovich liftings to generate liftings that are not induced by lax extensions.
\section{Behavioural Distance}
\label{sec:bd}

One main motivation for lifting functors to metric spaces was to obtain coalgebraic notions of behavioural distance~\cite{BBKK18,WS20}. Indeed,
\emph{every} functor \(\ftF \colon \Cats{\V}\to \Cats{\V}\) gives rise to a notion of distance on a \(\ftF\)-coalgebras:
\begin{defn}\cite{ForsterEA23}
	Let \((X,a,\alpha)\) be a coalgebra for a functor \(\ftF \colon \Cats{\V} \to \Cats{\V}\).
	The \df{behavioural distance} \(\bd_\alpha^\ftF(x,y)\) of \(x,y \in X\) is
	\begin{align}\label{eq:bd-form}
		\bd_\alpha^\ftF(x,y) = \bigvee \{b(f(x),f(y)) \mid f \colon (X,a,\alpha) \to (Y,b,\beta) \in \Coalg{\ftF}\}.
	\end{align}
\end{defn}
\noindent Notice the analogy with the standard notion of behavioural
equivalence: Two states are behaviourally equivalent if they can be
made equal under some coalgebra morphism; and according to the above
definition,  two states in a metric coalgebra have low behavioural distance if they can be
made to have low distance under some coalgebra morphism.

Kantorovich liftings and lax extensions are key ingredients in mentioned alternative coalgebraic approaches to behavioural distance on \(\SET\)-based coalgebras.
Let \(\ftF \colon \SET \to \SET\) be a functor.
A Kantorovich lifting \(\ftF^\Lambda \colon \Cats{\V} \to \Cats{\V}\) induces a notion of \emph{behavioural distance} on an \(\ftF\)-coalgebra \(\alpha \colon X \to \ftF X\) as the greatest \(\V\)-categorical structure \((X,a)\) that makes \(\alpha\) a \(\V\)-functor of type \((X,a) \to \ftF^\Lambda (X,a)\) \cite{BBKK18,KKK+21}.
From Theorem~\ref{p:53} and \cite[Proposition~12]{ForsterEA23} (generalized to \(\Cats{\V}\), with the same proof), we obtain that this distance coincides with behavioural distance as defined above.
On the other hand, every lax extension \(\eF \colon \Rels{\V} \to \Rels{\V}\) of~\(\ftF\) also induces a \emph{behavioural distance} on an \(\ftF\)-coalgebra \(\alpha \colon X \to \ftF X\) as the greatest \emph{simulation} on \(\alpha\) \cite{Rut98,Wor00,Gavazzo18,WS20},
i.e.\ the greatest \(\V\)-relation \(s \colon X \relto X\) such that \(\alpha \cdot s \leq \eF s \cdot \alpha\). It follows by routine calculation that this distance coincides with the distance defined via the lifting induced by the lax extension and, hence, Theorem~\ref{p:57} ensures that, if we start with a collection of monotone predicate liftings, then the corresponding Kantorovich extension and Kantorovich lifting yield the same notion of behavioural distance.
This allows including the approach to behavioural distance via lax extensions in the categorical framework for \emph{indistinguishability} introduced recently by Komorida et al. \cite{KKK+21}.
On the other hand, there are notions of behavioural distance defined via Kantorovich liftings that do not arise via lax extensions.
Indeed, it has been shown that the neighbourhood functor \(\ftN \colon \SET \to \SET\) does not admit a lax extension to \(\REL\) that preserves converses (\(\eF (r^\circ) = (\eF r)^\circ\)) whose (\(\two\)-valued) notion of behavioural distance coincides with behavioural equivalence \cite[Theorem~12]{MV15}.
However, from \cite[Theorem~34, Proposition~A.6]{ForsterEA23} (see also \cite{HKP09}), we can conclude that the (\(\two\)-valued) notion of behavioural distance defined by the canonical Kantorovich lifting  of \(\ftN\) to \(\EQ\) w.r.t. to the predicate lifting induced by the identity natural transformation \(\ftN \to \ftN\) coincides with behavioural equivalence.
(It is easy to see that Marti and Venema's result holds even if one allows lax extensions of \(\ftN\) that do not preserve converses, and that the situation remains the same in the asymmetric case.) 

\section{Expressivity of Quantitative Coalgebraic Logics}\label{sec:log}

We proceed to connect the characterization of Kantorovich functors
with existing expressivity results for quantitative coalgebraic logic,
focusing from now on on symmetric \(\V\)-categories. Therefore, we
interpret the \(\V\)-categorical notions and results also with
\(\Cats{\V}_{\sym}\) instead of \(\Cats{\V}\) and \(\V_{s}\) instead
of \(\V\).

We recall a variant \cite{ForsterEA23} of
(quantitative) coalgebraic
logic~\cite{Pat04,Sch08,CKP+11,KonigMikaMichalski18,WS20} that follows the
paradigm of interpreting modalities via predicate liftings, in this case of
$\V$-valued predicates for a \(\Cats{\V}\)-functor (Section~\ref{sec:pl}). Let \(\Lambda\) be a \emph{set} of finitary predicate liftings for a functor \(\ftF \colon \Cats{\V}_\sym \to \Cats{\V}_\sym\).
The syntax of \emph{quantitative coalgebraic modal logic} is then defined by the grammar
\begin{flalign*}
	 &  & \phi\Coloneqq \top \mid \phi_1 \vee \phi_2 \mid \phi_1 \wedge \phi_2 \mid u \otimes \phi \mid \hom_s(u,\phi) \mid \lambda(\phi_1,\dots,\phi_n) &  & (u \in \V, \lambda \in \Lambda)
\end{flalign*}
where $\Lambda$ is a set of \emph{modalities} of finite arity, which we identify, by abuse of notation, with the given set \(\Lambda\) of predicate liftings.
We view all other connectives as propositional operators.
Let $\calL(\Lambda)$ be the set of modal formulas thus defined.

The semantics is given by assigning to each formula \(\phi \in \calL(\Lambda)\) and each coalgebra \(\alpha \colon X \to \ftF X\) the \emph{interpretation} of~$\phi$ over~$\alpha$,
i.e.\ the \(\V\)-functor \(\Sema{\phi} \colon X \to \V\) recursively defined as follows:
\begin{itemize}[wide]
	\item for \(\phi = \top\), we take \(\Sema{\top}\) to be the \(\V\)-functor given by the constant map into~\(\top\);
	\item for an $n$-ary propositional operator~$p$, we put $\Sema{p(\phi_1, \ldots, \phi_n)} = p (\Sema{\phi_1}, \ldots, \Sema{\phi_n})$, with~$p$ interpreted using the lattice structure of~$\V$ and the \(\V\)-categorical structure \(\hom_s\) of \(\V_s\), respectively, on the right-hand side;
	\item for  \(n\)-ary $\lambda \in \Lambda$, we put $\Sema{\lambda(\phi_1, \ldots, \phi_n)} = \lambda (\langle\Sema{\phi_1}, \ldots, \Sema{\phi_n}\rangle) \cdot \alpha$, where \(\langle \Sema{\phi_1}, \ldots, \Sema{\phi_n}\rangle\) denotes the  \(\V\)-functor \((X,a)\to\V^n\) canonically determined by \(\Sema{\phi_1}, \ldots, \Sema{\phi_n}\).
\end{itemize}
We then obtain a notion of logical distance:
\begin{defn}
	Let \(\Lambda\) be a set of predicate liftings for a functor \(\ftF \colon \Cats{\V} \to \Cats{\V}\).
	The \df{logical distance} \(ld_\alpha^\Lambda\) on an \(\ftF\)-coalgebra \((X,a,\alpha)\) is the initial structure on \(X\) w.r.t.\ the structured cone of all maps
	\(
	\Sema{\phi} \colon X \to \ftII{(\V,\hom_s)}
	\)
	with \(\phi \in \calL(\Lambda)\).
	More explicitly, for all \(x,y \in X\),
	\begin{displaymath}\textstyle
		ld_\alpha^\Lambda(x,y) = \bigwedge \{\hom_s(\Sema{\phi}(x),\Sema{\phi}(y)) \mid \phi \in \calL(\Lambda)\}.
	\end{displaymath}
\end{defn}
In the remainder of the paper, we establish criteria under which a \(\Cats{\V}_\sym\)-functor admits a set of predicate liftings for which logical and behavioural distances coincide.
Recall that a (quantitative) coalgebraic logic is \df{expressive} if \(ld_\alpha^\Lambda \leq \bd_\alpha^\ftF\), for every \(\ftF\)-coalgebra~\((X,\alpha)\).
(It is easy to show that the reverse inequality holds universally \cite[Theorem~16]{ForsterEA23}).

Existing expressivity results for quantitative coalgebraic logics for \(\SET\)-functors depend crucially on Kantorovich liftings (e.g.~\cite{WS20,WS21,KKK+21,ForsterEA23}).
However, it has been shown \cite{ForsterEA23} that the Kantorovich property can be usefully detached from the notion of functor lifting.
\begin{defn}
	Let \(\Lambda\) be a class of predicate liftings for a functor \(\ftF \colon \Cats{\V} \to \Cats{\V}\).
	The functor \(\ftF\) is \df{\(\Lambda\)-Kantorovich} if for every \(\V\)-category \(X\), the cone of all \(\V\)-functors
	\(
	\lambda(f) \colon \ftF X \to \V
	\),
	with \(\lambda\in \Lambda\) \(\kappa\)-ary and \(f \in \Cats{\V}(X,\V^\kappa)\), is initial.
	A functor \(\ftF \colon \Cats{\V} \to \Cats{\V}\) is said to be \df{Kantorovich} if it is \(\Lambda\)-Kantorovich for some class \(\Lambda\) of predicate liftings for \(\ftF\).
\end{defn}
Clearly, every Kantorovich lifting of a \(\SET\)-functor to \(\Cats{\V}\) w.r.t.\ a class \(\Lambda\) of predicate liftings is \(\Lambda\)-Kantorovich.
Moreover, Theorem~\ref{p:53} is easily generalized to Kantorovich functors.

\begin{theorem}\label{thm:kantorovich-initial}
	A \(\Cats{\V}\)-functor is Kantorovich iff it preserves initial morphisms.
\end{theorem}

\begin{theorem}
	\label{p:18}
	A \(\Cats{\V}_\sym\)-functor is Kantorovich iff it preserves initial morphisms.
\end{theorem}
\begin{example}
	\label{p:14}
	\begin{enumerate}[wide]
		\item The inclusion functor \(\Cats{\V}_{\sym,\sep} \hookrightarrow \Cats{\V}_\sym\) has a
		      left adjoint \((-)_q \colon \Cats{\V}_\sym \to \Cats{\V}_{\sym,\sep}\) that quotients every \(X\) by its natural preorder, which for  symmetric~$X$ is an equivalence, and gives rise to a Kantorovich functor on \(\Cats{\V}_\sym\).
		\item \label{p:17} Given a bounded-by-\(1\) pseudometric space \((X,d)\), i.e. an object of \(\Cats{[0,1]_\oplus}_\sym \simeq \PBMET\), the \df{Prokhorov distance} \cite{Pro56} for probability measures on the measurable space of Borel sets of \((X,d)\) is defined by \(d^P(\mu,\upsilon) = \inf\{\epsilon > 0 \mid \mu(A) \leq \upsilon (A^\epsilon) + \epsilon \text{ for all Borel sets $A\subseteq X$}\}\), where \(A^\epsilon = \{x \in X \mid \inf_{y \in A} d(x,y) \leq \epsilon\}\).
		      It is straightforward to verify that this distance defines a \(\PBMET\)-functor (which acts on morphisms by measuring preimages) that preserves isometries and, therefore, it is Kantorovich.
		\item \label{p:15} For every \(\V\)-category \((X,a)\), the functor
		      \(
		      (X,a)\times- \colon\Cats{\V}\to\Cats{\V}
		      \)
		      is Kantorovich.
		      If the underlying lattice of \(\V\) is Heyting, then under certain conditions this functor has a right adjoint \cite{CH06,CHS09}
		      which is Kantorovich as well.
		      Here, for \(X=(X,a)\) exponentiable, the right adjoint \((-)^{X}\) of \(X\times-\) sends a
		      \(\V\)-category \(Y=(Y,b)\) to the \(\V\)-category \(Y^{X}=(Y^{X},c)\) with
		      underlying set
		      \(
		      \{\text{all \(\V\)-functors \((1,k)\times(X,a)\to(Y,b)\)}\}
		      \)
		      and, for \(h,k\in Y^{X}\),
		      \begin{displaymath}\textstyle
			      c(h,k)=\bigwedge_{x_{1},x_{2}\in X}b(h(x_{1}),k(x_{2}))^{a(x_{1},x_{2})},
		      \end{displaymath}
		      where \((-)^{u}\colon\V\to\V\) denotes the right adjoint of \(u\wedge- \colon\V\to\V\).
		      For a \(\V\)-functor \(f \colon (Y_{1},b_{1})\to(Y_{2},b_{2})\), the
		      \(\V\)-functor \(f^{X}\colon (Y_{1}^{X},c_{1})\to (Y_{2}^{X},c_{2})\) sends
		      \(h\in Y_{1}^{X}\) to~\(f\cdot h\).
	\end{enumerate}
\end{example}

\noindent To ensure that a Kantorovich functor is represented by finitary predicate liftings, we need to impose a size constraint:

\begin{defn}
	A functor~$\ftF\colon\Cats{\V}_\sym \to\Cats{\V}_\sym$ is \df{\(\omega\)-bounded} if for every symmetric $\V$-category~$X$ and every $t\in\ftF X$, there exists a finite subcategory $X_0\subseteq X$ and $t'\in\ftF X_0$ such that $t=\ftF i(t')$ where~$i$ is the inclusion $X_0\to X$.
\end{defn}

\begin{example}
	Every lifting of a finitary \(\SET\)-functor to \(\Cats{\V}_\sym\) is \(\omega\)-bounded.
\end{example}
\begin{proposition}
	\label{p:16}
	Let \(\ftF \colon \Cats{\V}_\sym \to \Cats{\V}_\sym\) be a Kantorovich functor.
	If \(\ftF\) is \(\omega\)-bounded, then \(\ftF\) is Kantorovich w.r.t. a set of finitary predicate liftings.
\end{proposition}

\noindent Finally, from \cite[Theorem~31]{ForsterEA23} we obtain:

\begin{corollary}
	Let \(\V\) be a finite quantale, and let \(\ftF \colon \Cats{\V}_\sym \to \Cats{\V}_\sym\) be a lifting of a finitary functor that preserves initial morphisms.
	Then there is a set \(\Lambda\) of predicate liftings for \(\ftF\) of finite arity such that the coalgebraic logic \(\calL(\Lambda)\) is expressive.
\end{corollary}
\begin{corollary}\label{cor:unit-hm}
	Let \(\ftF \colon \PBMET \to \PBMET\) be a functor that preserves isometries, is locally non-expansive, and admits a dense \(\omega\)-bounded subfunctor.
	Then there is a set \(\Lambda\) of predicate liftings for \(\ftF\) of finite arity such that the coalgebraic logic \(\calL(\Lambda)\) is expressive.
\end{corollary}

These instantiate to results on concrete system types, e.g. ones induced by (sub)functors listed in Example \ref{p:14}, such as probabilistic transition systems equipped with a behavioural distance induced by the functor that sends a bounded metric space \(X\) to the subspace of the space of all probability measures on \(X\) equipped with the Prokhorov distance (see Example \ref{p:14}(\ref{p:17})) determined by the closure of the set of finitely supported probability measures.

\section{Conclusions and Future Work}

Quantitative coalgebraic Hennessy-Milner
theorems~\cite{KM18,WS20,ForsterEA23} assume that the functor (on
metric spaces) describing the system type is \emph{Kantorovich},
i.e.~canonically induced by a suitable choice of -- not necessarily
monotone -- predicate liftings, which then serve as the modalities of
a logic that characterizes behavioural distance. We have shown as one
of our main results that a functor on (quantale-valued) metric spaces
is Kantorovich iff it preserves initial morphisms
(i.e.~isometries). As soon as such a functor additionally adheres to
the expected size and continuity constraints (which replace the
condition of finite branching found in the classical Hennessy-Milner
theorem for labelled transition systems), one thus has a logical
characterization of behavioural distance in coalgebras for the
functor, in the sense that behavioural distance equals logical
distance.

In fact we have shown that \emph{every} functor on metric spaces can
be captured by a generalized form of predicate liftings where the
object of truth values may change along the lifting. A simple example
is the discretization functor, which is characterized by a predicate
lifting in which the truth value object for the input predicates is
equipped with the indiscrete pseudometric, so that the lifting accepts
\emph{all} predicates instead of only non-expansive ones. This hints
at a perspective to design heterogeneous modal logics that
characterize behavioural distance for such functors, with modalities
connecting different types of formulas (e.g.\ non-expansive
vs.~unrestricted), which we will pursue in future work. One
application scenario for such a logic are behavioural distances on
probabilistic systems involving total variation distance, which may be
seen as a composite of the usual probabilistic Kantorovich functor and
the discretization functor.

\bibliographystyle{splncs04}

\providecommand{\noopsort}[1]{}

\clearpage
\appendix
\section{Appendix: Omitted Proofs and Details}

\subsection{Proof of~\autoref{p:54}}

\begin{lemma}
	\label{p:73}
	Let \(\ftbF \colon \Cats{\V} \to \Cats{\V}\) be a lifting of a functor \(\ftF \colon \SET \to \SET\), \(X\) be a \(\V\)-category and \(\Lambda\) be a class of natural transformations
	\[
		\lambda \colon \Cats{\V}(-,A_\lambda) \longrightarrow \Cats{\V}(\ftbF -, \V).
	\]
	If the cocone
	\[
		\big(\lambda_{X} \colon \Cats{\V}(X,A_\lambda) \longrightarrow \Cats{\V}(\ftbF X, \V)\big)_{\lambda \in \Lambda}
	\]
	is jointly epic, then the topological lifting w.r.t. the class of all natural transformations
	\[
		\lambda \colon \Cats{\V}(-,A_\lambda) \longrightarrow \SET(\ftF\ftII{-}, \ftII{\V})
	\]
	that factor as
	\begin{displaymath}
		\begin{tikzcd}
			\Cats{\V}(-,A_\lambda) & \Cats{\V}(\ftbF -,\V) \\
			& \SET(\ftF\ftII{-}, \ftII{\V})
			\ar[from=1-1, to=1-2, "\lambda"]
			\ar[from=1-1, to=2-2, "\lambda"']
			\ar[from=1-2, to=2-2, "\ftII{-}_{-,\V}"]
		\end{tikzcd}
	\end{displaymath}
	for some \(\lambda \in \Lambda\) coincides with \(\ftbF\) on \(X\).
\end{lemma}
\begin{proof}
	It follows by hypothesis that the topological lifting maps the \(\V\)-category \(X\) to the domain of the initial lift of the structured cone
	\[
		(\ftII{f} \colon \ftF\ftII{X} \longrightarrow \ftII{\V})_{f \in \Cats{\V}(\ftbF X, \V)}.
	\]
	Therefore, the claim is consequence of the fact that the cone
	\(\Cats{\V}(\ftbF X, \V)\) is initial (see Remark~\ref{d:prop:2}).
\end{proof}

The Yoneda lemma guarantees that there is an epi-cocone of natural transformations
\[
	\lambda \colon \Cats{\V}(-,A_\lambda) \longrightarrow \Cats{\V}(\ftbF -, \V)
\]
determined by every \(\V\)-category \(A\) and every \(\V\)-functor \(\lambda \colon \ftbF A \to \V\).
\qed

\subsection{Details of Example~\autoref{p:21}(\ref{p:23})}

Consider the discrete preordered set with two elements \((2=\{0,1\},1_2)\) and the preordered set \((3=\{0,1,2\},\vee)\) with three elements generated by \(2 \leq 1\) and \(2 \leq 0\).
Then, the inclusion \((2,1_2) \hookrightarrow (3,\vee)\) is initial, however its image under the lifting it is not, since the lifting acts as identity on \((2,1_2)\) but sends \((3,\vee)\) to the indiscrete preordered set with three elements.
To see that this is a topological lifting of identity functor on \(\SET\) constructed from the \(\two\)-valued identity predicate lifting for \(\ftId \colon \SET \to \SET\) by choosing \(A=(2,1_2)\), note that for every preordered set \((X,r)\), \(x\) is equivalent to \(y\) w.r.t to the smallest equivalence relation that contains \(r\) if and only if there is a zigzag between \(x\) and \(y\) in \((X,r)\) if and only if for every monotone map \(f \colon (X,r) \to (2,1_2)\), \(f(x) = f(y)\).

\subsection{Proof of~\autoref{p:51}}

Let \(\ftbF \colon \Cats{\V} \to \Cats{\V}\) be a lifting of \(\ftF\) and \(S\) be a set of predicate liftings for \(\ftF\).
We claim that \(\ftP(\ftbF) \supseteq S\) iff \(\ftbF \leq \ftF^S\).
Suppose that \(\ftP(\ftbF) \supseteq S\).
Then, since every \(\kappa\)-ary predicate lifting in \(S\) is compatible with \(\ftbF\), for every \(\V\)-functor \(f \colon (X,a) \to \V^\kappa\), the map
\begin{displaymath}
	\begin{tikzcd}
		\ftbF(X,a) & \ftbF(\V^\kappa)& \V
		\ar[from=1-1, to=1-2, "\ftF f"']
		\ar[from=1-2, to=1-3, "\lambda(1_{\V^\kappa})"']
		\ar[from=1-1, to=1-3, bend left, "\lambda(f)"]
	\end{tikzcd}
\end{displaymath}
is a \(\V\)-functor.
Therefore, \(\ftbF(X,a) \leq \ftF^S (X,a)\).

On the other hand, suppose that \(\ftbF \leq \ftF^S\).
By definition of Kantorovich lifting, every predicate lifting in \(S\) is compatible with \(\ftF^S\).
Hence, for every \(\lambda \colon \ftF^S(\V^\kappa) \to \V\) determined by a \(\kappa\)-ary predicate lifting \(\lambda\) in \(S\),
\begin{displaymath}
	\begin{tikzcd}
		\ftbF(\V^\kappa) & \ftF^S(\V^\kappa)& \V
		\ar[from=1-1, to=1-2, "1_{\ftF\V^\kappa}"']
		\ar[from=1-2, to=1-3, "\lambda"']
		\ar[from=1-1, to=1-3, bend left, "\lambda"]
	\end{tikzcd}
\end{displaymath}
is a \(\V\)-functor.
\qed

\subsection{Proof of~\autoref{p:53}}

Firstly, we show that every Kantorovich lifting of a functor \(\ftF \colon \SET \to \SET\) preserves initial morphisms.

Let \(i \colon (X,a) \to (Y,b)\) be an initial morphism,
\(\Lambda\) a class of predicate liftings for \(\ftF\),
\(j \colon (Z,c) \to \ftF^\Lambda(Y,b)\) a \(\V\)-functor, and
\(h \colon Z \to \ftF X\) a map such that \(j = \ftF i \cdot h\).

By definition of \(\ftF^\Lambda\) it is sufficient to show that for every
\(\kappa\)-ary predicate lifting for \(\ftF\) in \(\Lambda\) and every \(\V\)-functor
\(f \colon (X,a) \to \V^\kappa\), \(\lambda(f) \cdot h\) is a \(\V\)-functor.
Since \(\V\) is injective in \(\Cats{\V}\) w.r.t. initial
morphisms, \(\V^\kappa\) is also injective in \(\Cats{\V}\) with respect
to initial morphisms.
Hence, for every \(\V\)-functor
\(f \colon (X,a) \to \V^\kappa\) there is a \(\V\)-functor
\(g \colon (Y,b) \to \V^\kappa\) such that \(f = g \cdot i\).
Consequently,
\[
	\lambda(f) \cdot h
	= \lambda(g \cdot i) \cdot h
	= \lambda(g) \cdot \ftF i \cdot h
	= \lambda(g) \cdot j
\]
is a \(\V\)-functor.

Secondly, we show that the converse statement holds. Suppose that \(\ftbF\)
is a lifting that preserves initial morphisms. We already know from
Theorem~\ref{p:51} that \(\ftbF \leq \ftF^{\ftP(\ftbF)}\). To prove that under
our assumption the reverse inequality also holds, let \((X,a)\) be a
\(\V\)-category and \(\kappa = |X|\). Then, the Yoneda embedding \((X,a)
\to [(X,a)^\op,\V]\) gives us an initial morphism

\begin{displaymath}
	\begin{tikzcd}
		(X,a) & {[(X,a)^\op,\V]} & \V^X & \V^\kappa.
		\ar[from=1-1, to=1-2]
		\ar[from=1-2, to=1-3, hookrightarrow]
		\ar[from=1-1, to=1-4, bend left, "\yoneda"]
		\ar[from=1-3, to=1-4, "\sim"]
	\end{tikzcd}
\end{displaymath}

Hence, since \(\ftbF\) preserves initial morphisms,
\(\ftbF \yoneda \colon \ftbF(X,a) \to \ftbF(\V^\kappa)\)
is initial. Now, let \(\ftP_\kappa(\ftbF)\) denote the set of all
\(\kappa\)-ary predicate liftings in \(\ftP(\ftbF)\).
Given that the cone of all \(\V\)-functors
\[
	\ftbF(\V^\kappa) \longrightarrow \V
\]
is initial and the composition of initial cones is initial, the cone
\[
	(\lambda(\yoneda) \colon \ftbF(X,a) \longrightarrow \V)_{\lambda \in \ftP_\kappa(\ftbF)}
\]
is initial. Therefore, \(\ftF^{\ftP(\ftbF)} (X,a) \leq \ftbF (X,a)\).
\qed

\subsection{Proof of~\autoref{p:57}}

To obtain this result it is useful to express the Kantorovich lifting in the language of \(\V\)-relations \cite{GoncharovEA22}.
In the sequel, given a \(\V\)-functor \(f \colon X \to \V^\kappa\), we denote by \(f^\flat \colon \kappa \relto X\) the \(\V\)-relation defined by \(f^\flat(k,x) = f(x)(k)\), for all \(k \in \kappa\) and \(x \in X\).
Also, we recall that a \(\V\)-relation \(r \colon X \relto Y\) is a distributor \(r \colon (X,a) \modto (Y,b)\) if \(r \cdot a \leq r\) and \(b \cdot r \leq r\).

\begin{proposition}
	\label{p:35}
	Let \((X,a)\) be a \(\V\)-category, \(\kappa\) a cardinal and
	\(f \colon X \to \V^\kappa\) a function. The following propositions are
	equivalent:
	\begin{enumerate}
		\item \(f \colon (X, a) \to \V^\kappa\) is a \(\V\)-functor;
		\item \(a \leq f^\flat \multimapdotinv f^\flat\);
		\item \label{p:36} \(a \cdot f^\flat = f^\flat\);
		\item \(f^\flat \colon (\kappa, 1_\kappa) \relto (X,a)\) is a \(\V\)-distributor.
	\end{enumerate}
\end{proposition}

\begin{corollary}
	\label{p:68}
	Let \(\calC = (f_i \colon X \to |\V^\kappa|)_{i \in I}\) be a structured cone.
	The initial  \(\V\)-category on \(X\) w.r.t. \(\calC\) is given by
	\[
		\bigwedge_{i \in I} {f_i}^\flat \multimapdotinv {f_i}^\flat.
	\]
\end{corollary}

\begin{proposition}
	\label{p:300}
	Let \(\lambda\) be a \(\kappa\)-ary \(\V\)-valued predicate lifting for a functor \(\ftF \colon \SET \to \SET\).
	The Kantorovich lifting \(\ftF^\lambda \colon \Cats{\V} \to \Cats{\V}\) of \(\ftF\) w.r.t. \(\lambda\) sends a \(\V\)-category \((X,a)\) to the \(\V\)-category \((\ftF X, \ftF^\lambda a)\), where
	\[
		\ftF^\lambda a
		= \bigwedge_{r \colon (\kappa,1_\kappa) \modto (X,a)} \lambda(r) \multimapdotinv \lambda(r).
	\]
\end{proposition}

Therefore, the Kantorovich lifting of a class of predicate liftings~\(\Lambda\) sends a
\(\V\)-category \((X,a)\) to the \(\V\)-category \((\ftF X, \ftF^\Lambda a)\), where
\(
\ftF^\Lambda a
= \bigwedge_{\lambda \in \Lambda} \ftF^\lambda a.
\)

Now, Theorem~\ref{p:57} is an immediate consequence of the following lemma.

\begin{lemma}
	\label{p:78}
	Let \(\lambda\) be a \(\kappa\)-ary \(\V\)-valued monotone predicate lifting for a functor \(\ftF \colon \SET \to \SET\), and let \((X,a)\) be a \(\V\)-category.
	Then,
	\[
		\bigwedge_{g \colon \kappa \relto X} \lambda(a \cdot g) \multimapdotinv \lambda(g)
		= \bigwedge_{r \colon (\kappa,1_\kappa) \modto (X, a)} \lambda(r) \multimapdotinv \lambda(r).
	\]
\end{lemma}
\begin{proof}
	Let \(g \colon \kappa \relto X\) be a \(\V\)-relation.
	Since \(1_X \leq a\), we have \(g \leq a \cdot g\).
	Also, given that \(a \cdot a \leq a\),
	we obtain that \(a \cdot g \colon (\kappa,1_\kappa) \modto (X, a)\) is a \(\V\)-distributor.
	Therefore, because \(\lambda\) is monotone,
	\[
		\lambda(a \cdot g) \multimapdotinv \lambda(a \cdot g) \leq
		\lambda(a \cdot g) \multimapdotinv \lambda(g).
	\]
	The other inequality is an immediate consequence of
	Proposition~\ref{p:35}(\ref{p:36}).
\end{proof}

\subsection{Proof of~\autoref{p:61}}

In the sequel, given a \(\V\)-relation \(r \colon X \relto Y\), we denote by \(r^\sharp \colon Y \to \V^X\) the function that sends each \(y \in Y\) to the function \(r(-,y) \colon X \to \V\).
Moreover, for every set \(S\), we denote the structure of the \(\V\)-category \(\V^S\) by \(h^S\).

\begin{proposition}
	\label{p:59}
	Every predicate lifting induced by a lax extension of a \(\SET\)-functor to \(\Rels{\V}\) is compatible with the lifting to \(\Cats{\V}\) induced by the lax extension.
\end{proposition}
\begin{proof}
	Let \(\lambda\) be a \(\kappa\)-ary \(\V\)-valued predicate lifting for a functor \(\ftF \colon \SET \to \SET\) that is induced by a lax extension \(\eF \colon \Rels{\V} \to \Rels{\V}\) of \(\ftF\).
	Then, by \cite[Theorem~3.11]{GoncharovEA22},
	\begin{align*}
		\eF h^\kappa \cdot \lambda(\ev_\kappa)
		 & = \eF h^\kappa \cdot \lambda(h^\kappa \cdot 1_\kappa^\sharp)                     \\
		 & = \eF h^\kappa \cdot \eF h^\kappa \cdot \ftF(1_\kappa^\sharp) \cdot \lambda(1_k) \\
		 & = \eF h^\kappa \cdot \ftF(1_\kappa^\sharp) \cdot \lambda(1_k)                    \\
		 & = \lambda (\ev_\kappa).
	\end{align*}
	Therefore, the claim follows from Propositions~\ref{p:35} and~\ref{p:2}.
\end{proof}

\begin{proposition}
	Let \(\ftF \colon \SET \to \SET\) be a functor.
	The monotone map \(\ftI \colon \LaxF \to \LiftF_\ftI\) is order-reflecting.
\end{proposition}
\begin{proof}
	Note that every \(\V\)-relation \(r \colon X \relto Y\) can be factorized as \(r = (r^\sharp)^\circ \cdot h^X \cdot 1_X^\sharp\).
\end{proof}

Let  \(\eF \colon \Rels{\V} \to \Rels{\V}\) be a lax extension of \(\ftF\).
Then, by Propositions~\ref{p:59} and~Remark~\ref{p:5}, \(\ftM\ftP\ftI(\eF)\)
contains the all predicate liftings induced \(\eF\).
Hence, by Theorem~\ref{d:thm:1},
\[
	\eF^{\ftM\ftP\ftI(\eF)} \leq \eF.
\]

On the other hand, let \(\ftbF \colon \Cats{\V} \to \Cats{\V}\) be a lifting of \(\ftF\) that preserves initial morphisms.
Then, by Theorem~\ref{p:57} and the fact that \(\ftM\ftP \colon \LiftFi \to \PredFm\) is left adjoint to \(\ftF^{(-)} \colon \PredFm \to \LiftFi\),
\[
	\ftbF \leq \ftF^{\ftM\ftP(\ftbF)} = \ftI(\eF^{\ftM\ftP(\ftbF)}).
\]
\qed

\subsection{Proof of~\autoref{p:19} and~\autoref{p:1}}

We will show~\autoref{p:19} as~\autoref{p:1} follows identically by taking into account \cite[Theorem~2.16]{GoncharovEA22} and~\cite[Theorem~4.1]{GoncharovEA22}.
We already observed that every lifting induced by a lax extension preserves initial morphisms, and the fact that it is also locally monotone follows immediately from L1 and L3 since \(f \leq g\) in \(\Cats{\V}((X,a),(Y,b))\) iff \(1_X \leq g^\circ \cdot b \cdot f\).

On the other hand, suppose that \(\ftbF \colon \Cats{\V} \to \Cats{\V}\) is a lifting of \(\ftF\) that preserves initial morphisms and is locally monotone.
Then, since every \(\V\)-functor is monotone, we have that every \(\kappa\)-ary \(\V\)-valued predicate lifting \(\lambda\) for \(\ftF\) compatible with \(\ftbF\) is monotone, as each \(X\)-component is given by the composite of monotone maps
\begin{displaymath}
	\begin{tikzcd}
		\Cats{\V}((X,1_X),\V^\kappa) & \Cats{\V}(\ftbF(X,1_X),\ftbF (\V^\kappa)) \\
		& \Cats{\V}(\ftbF((X,1_X),\V).
		\ar[from=1-1, to=1-2, "\ftF (-)"]
		\ar[from=1-2, to=2-2, "\lambda(1_{\V^\kappa}) \cdot -"]
		\ar[from=1-1, to=2-2, "\lambda_X"']
	\end{tikzcd}
\end{displaymath}
Therefore, by Theorem~\ref{p:57} and the proof of Theorem~\ref{p:53}, we conclude that \(\ftbF = \eF^{\ftP(\ftbF)}\).
\qed

\subsection{Details of~\autoref{sec:bd}}

\begin{proposition}
	\label{p:501}
	Let \(\ftbF \colon \Cats{\V} \to \Cats{\V}\) be a lifting of a functor \(\ftF \colon \SET \to \SET\) that preserves initial morphisms and corestricts to \(\Cats{\V}_\sym\).
	Then, for every \(\ftF\)-coalgebra \(\alpha \colon X \to \ftF X\), \(\bd_\alpha^\ftbF\) is symmetric.
\end{proposition}
\begin{proof}
	Let \(\alpha \colon X \to \ftF X\) be an \(\ftF\)-coalgebra.
	Since the lifting preserves initial morphisms, then by \cite[Proposition~12]{ForsterEA23} (the same proof holds for the non-symmetric case), \(\bd_\alpha^\ftF\) is the greatest \(\V\)-categorical structure that makes \(\alpha \colon (X,\bd_\alpha^\ftF) \to \ftbF(X,\bd_\alpha^\ftF)\) a \(\V\)-functor.
	In particular, this means that \(\alpha \colon (X,\bd_\alpha^\ftF) \to \ftbF (X,\bd_\alpha^\ftF)\) is initial.
	Therefore, as \(\ftbF\) corestricts to \(\Cats{\V}_\sym\) and this category is closed under initial cones in \(\Cats{\V}\), we conclude that \(\bd_\alpha^\ftF\) is symmetric.
\end{proof}

\begin{proposition}
	\label{p:502}
	Let \(\eF \colon \Rels{\V} \to \Rels{\V}\) be a lax extension of a functor \(\ftF \colon \SET \to \SET\) and let \(\alpha \colon X \to \ftF X\) be an \(\ftF\)-coalgebra.
	Every symmetric \(\eF\)-simulation on \(\alpha\) is also a \(\eF^s\)-(bi)simulation.
\end{proposition}
\begin{proof}
	Let \(s\) be a symmetric \(\eF\)-simulation on \(\alpha\).
	Then,
	\begin{align*}
		                    & \alpha \cdot s \leq \eF s \cdot \alpha                           \\
		\Longleftrightarrow & s^\circ \cdot \alpha^\circ \leq \alpha^\circ \cdot (\eF s)^\circ \\
		\Longrightarrow     & s \cdot \alpha^\circ \leq \alpha^\circ \cdot (\eF s)^\circ       \\
		\Longleftrightarrow & \alpha \cdot s \leq (\eF s)^\circ \cdot \alpha.
	\end{align*}
	Therefore, \(\alpha \cdot s \leq (\eF s \wedge (\eF s)^\circ) \cdot \alpha = \eF^s s \cdot \alpha\), since the map \(-\cdot \alpha\) preserves infima.
\end{proof}

\subsection{Details of \autoref{p:14}(\ref{p:17})}

The usual proof shows that the Prokhorov distance defines a pseudometric (e.g. \cite[p.27]{Hub81}), and to see that ``measuring preimages'' define non-expansive maps w.r.t. this pseudometric, note that for every non-expansive map \(f \colon X \to Y\),
\(
(f^{-1}B)^\epsilon \subseteq f^{-1}(B^\epsilon),
\)
for all \(B \subseteq Y\) and \(\epsilon \in [0,1]\).
Therefore, the Prokhorov pseudometric defines a functor \(\ftD^P \colon \PBMET \to \PBMET\).
Now, suppose that \(f \colon X \to Y\) is an isometry.
Let  \(\mu, \upsilon\) be probablity measures on \(X\), and let \(\epsilon \in (0,1]\) such that for every Borel set \(B \subseteq Y\), \(\mu(B) \leq \upsilon(B^\epsilon) + \epsilon\).
Consider a Borel set \(A \subseteq X\).
Then, \(\overline{fA} \subseteq Y\) is a Borel set which, by hypothesis, yields \(\mu(A) \leq \upsilon(f^{-1}(\overline{fA})^\epsilon) + \epsilon\).
By \cite[Lemma~3.4]{Hub81}, \({(\overline{fA})}^\epsilon = ({fA})^\epsilon\), and since \(f\) is initial \(f^{-1}(fA)^\epsilon = A^\epsilon\).
Hence, \(\mu(A) \leq \upsilon(A^\epsilon) + \epsilon\).
Therefore, \(d_X^P(\mu,\upsilon) \leq d_Y^P(\ftD f(\mu),\ftD f(\upsilon))\).

\subsection{Details of \autoref{p:14}(\ref{p:15})}

For every \(\V\)-category \((X,a)\), the functor
\(
(X,a)\times- \colon\Cats{\V}\to\Cats{\V}
\)
is Kantorovich. Under certain conditions, this functor has a right adjoint
which is Kantorovich as well. To see that, assume now that the underlying
lattice of \(\V\) is Heyting, and we denote the right adjoint of
\(u\wedge- \colon\V\to\V\) by \((-)^{u}\colon\V\to\V\). It has been shown
\cite{CH06,CHS09} that \((X,a)\times-\) is left adjoint if and only if, for
all \(x,z\in X\) and \(u,v\in\V\),
\begin{displaymath}
	\bigvee_{y\in Y}(a(x,y)\wedge u)\otimes(a(y,z)\wedge v)\geq
	a(x,z)\wedge(u\otimes v).
\end{displaymath}
For such \(X=(X,a)\), the right adjoint \((-)^{X}\) of \(X\times-\) sends a
\(\V\)-category \(Y=(Y,b)\) to the \(\V\)-category \(Y^{X}=(Y^{X},c)\) with
underlying set
\begin{displaymath}
	\{\text{all \(\V\)-functors \((1,k)\times(X,a)\to(Y,b)\)}\}
\end{displaymath}
and, for \(h,k\in Y^{X}\),
\begin{displaymath}
	c(h,k)=\bigwedge_{x_{1},x_{2}\in X}b(h(x_{1}),k(x_{2}))^{a(x_{1},x_{2})}.
\end{displaymath}
For a \(\V\)-functor \(f \colon (Y_{1},b_{1})\to(Y_{2},b_{2})\), the
\(\V\)-functor \(f^{X}\colon (Y_{1}^{X},c_{1})\to (Y_{2}^{X},c_{2})\) sends
\(h\in Y_{1}^{X}\) to \(f\cdot h\), moreover, if \(f\) is initial, then, for
all \(h,k\in Y_{1}^{X}\),
\begin{align*}
	c_{1}(h,k)
	 & =\bigwedge_{x_{1},x_{2}\in X}b_{1}(h(x_{1}),k(x_{2}))^{a(x_{1},x_{2})}       \\
	 & =\bigwedge_{x_{1},x_{2}\in X}b_{2}(f(h(x_{1})),f(k(x_{2})))^{a(x_{1},x_{2})}
	=c_{2}(f^{X}(h),f^{X}(k)).
\end{align*}

\subsection{Proof of~\autoref{thm:kantorovich-initial} and~\autoref{p:18}}
See proof of \autoref{p:53}.

\subsection{Proof of~\autoref{p:16}}

\begin{lemma}
	\label{p:20}
	Let \(f \colon (X,a) \to (Z,c)\), \(g \colon (Y,b) \to (Z,c)\) and \(h \colon (X,a) \to (Y,b)\) be \(\V\)-functors such that \(f = g \cdot h\).
	If \(f\) is initial, then for all \(x,x' \in X\), \(a(x,x') = b(h(x),h(x'))\).
\end{lemma}
\begin{proof}
	Let \(x,x' \in X\).
	Then,
	\(a(x,x') \leq b(h(x),h(x')) \leq c(g \cdot h (x),g \cdot h(x')) = c(f(x),f(x')) = a(x,x').\)
	\qed
\end{proof}

Let \(\ftF \colon \Cats{\V}_\sym \to \Cats{\V}_\sym\) be an \(\omega\)-bounded functor that preserves initial morphisms, and let \(\Lambda\) be the set of all finitary predicate liftings for \(\ftF\).
We claim that \(\ftF\) is \(\Lambda\)-Kantorovich.
Suppose that \(X\) is a symmetric \(\V\)-category and \(\fx,\fy\) are elements of \(\ftF X\).
Then, since \(\ftF\) is \(\omega\)-bounded, there is a finite subcategory \(i \colon X_0 \hookrightarrow X\) and \(\fx',\fy' \in \ftF X_0\) such that \(\fx = \ftF i (\fx')\) and \(\fy = \ftF i(\fy')\).
Moreover, given that \(i\) is initial and for every cardinal \(\kappa\) the \(\V\)-category \(\V_s^\kappa\) is injective in \(\Cats{\V}_\sym\) w.r.t. to initial morphisms, every \(\V\)-functor \(f \colon X_0 \to \V^\kappa\) can be factorized as \(\overline{f} \cdot i\), for some \(\V\)-functor \(\overline{f} \colon X \to \V_s^\kappa\).
Hence, for every \(\kappa\)-ary \(\lambda\) in \(\Lambda\), and every \(\V\)-functor \(f \colon X_0 \to \V_s^\kappa\), \(\lambda(f) = \lambda(\overline{f}) \cdot \ftF i\).
Therefore, as the cone of all \(\V\)-functors \(\lambda(f) \colon \ftF X_0 \to \V_s\) determined by each \(\kappa\)-ary \(\lambda\) in \(\Lambda\) and every \(\V\)-functor \(f \colon X_0 \to \V^\kappa\) is initial (see the proof of Theorem~\ref{p:53}), by Lemma~\ref{p:18} we obtain
\begin{align*}
	\ftF X(\fx,\fy) & = \bigwedge_{\lambda \in \Lambda}\bigwedge_{\overline{f} \colon X \to \V_s^{\ar(\lambda)}} \hom_s(\lambda(\overline{f})(\fx),\lambda(\overline{f})(\fy)) \\
	                & \geq \bigwedge_{\lambda \in \Lambda}\bigwedge_{g \colon X \to \V_s^{\ar(\lambda)}} \hom_s(\lambda(g)(\fx),\lambda(g)(\fy)),
\end{align*}
where \(\ar(\lambda)\) denotes the arity of the predicate lifting \(\lambda\).
\end{document}